\documentclass[11pt]{article}
\usepackage{amssymb,latexsym,amsmath,amsbsy,amsthm,amsxtra,amsgen,graphicx,dsfont,color}
\newfont{\msbm}{msbm10 at 11pt}
\newcommand {\R} {\mbox{\msbm R}}

\newcommand {\N} {\mbox{\msbm N}}

\newcommand {\1} {\mathds{1}}

\def\Var{\textup{Var}}

\def\eps{\varepsilon}

\newtheorem{Theo}{Theorem}
\newtheorem{Lemma}[Theo]{Lemma}
\newtheorem{Cor}[Theo]{Corollary}
\newtheorem{Prop}[Theo]{Proposition}

\begin{document}

\title{Mutation timing in a spatial model of evolution}
\author{Jasmine Foo\thanks{Supported in part by NSF Grant DMS-1349724 and the Fulbright Foundation}, Kevin Leder\thanks{Supported in part by NSF Grant CMMI-1552764 and the Fulbright Foundation}, and Jason Schweinsberg\thanks{Supported in part by NSF Grant DMS-1707953}}

\maketitle

\begin{abstract}
Motivated by models of cancer formation in which cells need to acquire $k$ mutations to become cancerous, we consider a spatial population model in which the population is represented by the $d$-dimensional torus of side length $L$.  Initially, no sites have mutations, but sites with $i-1$ mutations acquire an $i$th mutation at rate $\mu_i$ per unit area.  Mutations spread to neighboring sites at rate $\alpha$, so that $t$ time units after a mutation, the region of individuals that have acquired the mutation will be a ball of radius $\alpha t$.  We calculate, for some ranges of the parameter values, the asymptotic distribution of the time required for some individual to acquire $k$ mutations.  Our results, which build on previous work of Durrett, Foo, and Leder, are essentially complete when $k = 2$ and when $\mu_i = \mu$ for all $i$.
\end{abstract}

\section{Introduction}
Cancer is widely thought to arise due to a series of oncogenic mutations accumulating in a cell.   Mathematical work on this subject goes back to the celebrated 1954 paper of Armitage and Doll \cite{armdoll}, who proposed a model in which a cell that has already acquired $k-1$ mutations receives a $k$th mutation at rate $\mu_k$.  They showed that for small $t$, the probability that such a cell receives its $k$th mutation in the time interval $[t, t + dt]$ is approximately $$\frac{\mu_1 \mu_2 \dots \mu_k t^{k-1}}{(k-1)!} \: dt.$$  They also examined data on 17 types of cancer and found that in many instances the cancer incidence rate increases proportional to a power of age, consistent with their multi-stage model.

More recently, the model of Armitage and Doll has been extended in multiple ways.  Some authors have incorporated cell division and death by considering a Moran-type model of $N$ cells in which each cell dies at rate one, at which time a cell is chosen at random from the population to divide into two.  For some results on the distribution of the time required for two mutations to appear in this model, see \cite{imkn, imn, ksn}.  Results on the distribution of the time required for $k$ mutations to appear were obtained in \cite{dss} for some ranges of values of the mutation rates $\mu_1, \dots, \mu_k$, and essentially complete results for the case when $\mu_i = \mu$ for all $i$ were established in \cite{sch08}.  Another extension of the model is to allow for the possibility that cells that have acquired several mutations along the pathway to cancer may have a selective advantage over other cells.  In this case the number of cells with these mutations will evolve like a supercritical branching process.  Results for the distribution of the time required for some cell to accumulate $k$ mutations in this setting were established by Durrett and Moseley \cite{dm10}, with further extensions in \cite{dflmm, dm11}.

Another extension, which is especially relevant for solid tumors, is to consider a model with spatial structure.  One of the earliest spatially explicit stochastic models was developed by Williams and Bjerknes \cite{wb72} who used an interacting particle system on $\mathbb{Z}^2$ to model the spread of cancer cells in epithelial tissue. In this model, each site of the lattice is occupied by either a healthy or cancer cell.  Healthy cells divide at rate 1, and cancer cells divide at rate $1+\beta$ for some positive $\beta$. Upon a cell division the daughter cell randomly replaces one of the four nearest lattice neighbors. This stochastic process would later become known to the probability community as the biased voter model. Using the gambler's ruin formula, Williams and Bjerknes were able to derive results on the probability that the tumor cell and its descendants would eventually take over the entire population. In addition, they made several conjectures about the shape of the mutant cell population (conditioned on survival) that were later disproven in \cite{moll72}. The work of Williams and Bjerknes motivated Bramson and Griffeath \cite{bg1, bg2} to establish a rigorous shape theorem for the biased voter model in 2 and higher dimensions. More recently Komarova \cite{k06} studied a model that is very similar to this in one dimension, and then Durrett and Moseley \cite{dm15} extended her work by calculating the asymptotic distribution of the time required for some individual to acquire two mutations in $d \geq 2$.  Durrett and Moseley considered the case in which cells with one mutation have close to the same fitness level as cells with no mutations.  Durrett, Foo, and Leder \cite{dfl16} performed a similar analysis for the case in which cells with one mutation do have a selective advantage.  See also \cite{mh11, mkmh11} for related work.

For some of their results, Durrett, Foo, and Leder \cite{dfl16} worked not with the biased voter model but with a simpler model with continuous space which, in view of the shape theorem proved by Bramson and Griffeath \cite{bg1, bg2}, should approximate the behavior of the biased voter model. In addition \cite{flr14, rlrlf16} used this model to make quantitative predictions regarding the cancer field effect.  In this paper, we will consider a slight variation of their model, which we now describe.  We consider the $d$-dimensional torus $[0, L]^d$, where $d \geq 1$, and let $N = L^d$ be the volume of the torus.  Each site on the torus will be assigned a type, representing the number of mutations carried by the individual at that site.  At time zero, all sites have type zero.  At the times and locations of a homogeneous Poisson process of rate $\mu_1$ per unit area, a mutation to type 1 occurs.  A region of type 1 individuals then grows outward from this point at rate $\alpha$ per unit time. This means that $t$ time units after the original mutation, the region of type 1 individuals will be a ball of radius $\alpha t$, which eventually expands to cover the entire torus.  Type 1 individuals acquire a second mutation at rate $\mu_2$ per unit area, causing a region of type 2 individuals to grow outward at rate $\alpha$ per unit time.  This process then continues indefinitely, with type $k-1$ sites acquiring a $k$th mutation at rate $\mu_k$ per unit area, producing a region of type $k$ individuals which grows outward at rate $\alpha$ per unit time.  We denote by $\sigma_k$ the first time at which some site has acquired $k$ mutations.

The model described here is essentially the model considered in \cite{dfl16}, except that here we consider only the ``successful" mutations, whereas the model in \cite{dfl16} also attempted to account for the mutations that quickly die out.  Because we are only modeling the successful mutations, the mutation rate $\mu_i$ corresponds to $u_i s$ in \cite{dfl16}.  Our parameter $\alpha$, which measures the rate at which a beneficial mutation spreads to neighboring sites, is called $c_d(s)$ in \cite{dfl16}.  Ralph and Coop \cite{rc10} also considered a spatial model very similar to the one studied here.

Our goal in this paper is twofold.  First, we calculate the asymptotic distribution of $\sigma_2$.  This builds on the work of Durrett, Foo, and Leder \cite{dfl16}, who calculated this distribution for some ranges of the parameter values.  Second, we extend these results by computing, for some ranges for the parameters, the asymptotic distribution of $\sigma_k$ for $k \geq 3$, which is relevant for types of cancer that have more than two stages in their development.  In this case, because of the complexity of the problem, while we state some of our results for general mutation rates $\mu_i$, our results are only essentially complete in the case when we assume all mutation rates are the same, so that $\mu_i = \mu$ for all $i$.  We always assume that mutations spread at rate $\alpha$, which is essentially equivalent to the assumption that the selective advantage of type $i$ individuals over type $i-1$ individuals is the same for all $i$.  One could consider a more general model in which a region of type $i$ individuals spreads at rate $\alpha_i$, but we do not pursue this extension here.  Note that it is only clear how to formulate this model when $\alpha_1 \geq \alpha_2 \geq \dots$ because if $\alpha_i > \alpha_{i-1}$, then eventually the type $i$ region could completely swallow the type $i-1$ region and become adjacent to regions of type $k-2$ or lower, presumably allowing it to expand faster.

In Section \ref{2sec}, we describe our results for the asymptotic distribution of $\sigma_2$, and we explain heuristically why these should results should be true.  We give similar heuristics for the asymptotic distribution of $\sigma_k$ for $k \geq 3$ in Section \ref{ksec}.  We then give mathematically precise statements of the results, as well as complete proofs of the results, in Section \ref{proofsec}.

\section{Waiting for two mutations: results and heuristics}\label{2sec}

We summarize here the asymptotic results as $N \rightarrow \infty$ for the distribution of $\sigma_2$, the time that it takes for some site to acquire two mutations.  The distribution of $\sigma_2$ depends on the values of the parameters $\mu_1$, $\mu_2$, $\alpha$ and $N$.  Note that it should be understood that these parameters depend on $N$, even though this dependence is not recorded in the notation.  There are 11 different behaviors that are possible, depending on the parameter values.  We will let $\gamma_d$ denote the volume of the unit ball in $\R^d$, which appears in several of the limit theorems.

Given sequences $(a_N)_{N=1}^{\infty}$ and $(b_N)_{N=1}^{\infty}$, the notation $a_N \ll b_N$ means $\lim_{N \rightarrow \infty} a_N/b_N = 0$, and $a_N \gg b_N$ means $\lim_{N \rightarrow \infty} a_N/b_N = \infty$.  Also, $a_N \asymp b_N$ means $0 < \liminf_{N \rightarrow \infty} a_N/b_N \leq \limsup_{N \rightarrow \infty} a_N/b_N < \infty$.  We use $\Rightarrow$ to denote convergence in distribution as $N \rightarrow \infty$.

\bigskip
\noindent {\bf Case 1}: ${\displaystyle \mu_1 \ll \frac{\alpha}{N^{(d+1)/d}}}$ and $\mu_2 \gg \mu_1$.
\bigskip

The time that it takes before the first mutation appears is exponentially distributed with rate $N \mu_1$.  Because $L = N^{1/d}$, and the maximum distance between any two points on the $d$-dimensional torus is $\sqrt{d} L/2$, the time required for a mutation to spread to the entire population (or fixate) is $(\sqrt{d} N^{1/d})/(2 \alpha)$.  Thus, when $N^{1/d}/\alpha \ll 1/(N\mu_1)$, which is equivalent to the assumption that $\mu_1 \ll \alpha/N^{(d+1)/d}$, the time required for the first mutation to spread to fixate, once it appears, is much less than the time that it takes for the mutation to appear.  When $\mu_2 \gg \mu_1$, the second mutation appears much faster than the first one.  The dominant waiting time is therefore the time to wait for the first mutation, and we have
$$N \mu_1 \sigma_2 \Rightarrow W,\hspace{.3in}W \sim \textup{Exponential}(1).$$

\bigskip
\noindent {\bf Case 2}: ${\displaystyle \mu_1 \ll \frac{\alpha}{N^{(d+1)/d}}}$ and $\mu_2 \ll \mu_1$.
\bigskip

As in Case 1, the first mutation fixates very soon after it appears.  This time, the waiting time for the second mutation is much longer, which means $$N \mu_2 \sigma_2 \Rightarrow W,\hspace{.3in}W \sim \textup{Exponential}(1).$$

\bigskip
\noindent {\bf Case 3}: ${\displaystyle \mu_1 \ll \frac{\alpha}{N^{(d+1)/d}}}$ and ${\displaystyle \frac{\mu_2}{\mu_1} \rightarrow c \in (0, \infty)}$.
\bigskip

As in Case 1, the first mutation fixates very soon after it appears.  This time, the waiting times for the first and second mutations are the same order of magnitude, so the limit distribution is a sum of independent exponential random variables.  Therefore, $$N \mu_1 \sigma_2 \Rightarrow W_1 + W_2,$$ where $W_1 \sim \textup{Exponential}(1)$, $W_2 \sim \textup{Exponential}(c)$, and $W_1$ and $W_2$ are independent.  The results for Case 1, Case 2, and Case 3 all follow from Theorem \ref{expwait} in Section \ref{proofsec} below, as explained in the paragraph following the statement of Theorem \ref{expwait}.

\bigskip
\noindent {\bf Case 4}: ${\displaystyle \mu_1 \gg \frac{\alpha}{N^{(d+1)/d}}}$ and ${\displaystyle \mu_2 \gg \frac{(N \mu_1)^{d+1}}{\alpha^d}}.$ 
\bigskip

When $\mu_1 \gg \alpha/N^{(d+1)/d}$, it takes longer for a mutation to fixate, once it has appeared, than it takes for a mutation to appear.  This means that many mutations will appear before the entire population has acquired a mutation.  A mutation that appears at a given time will have grown to size $\gamma_d (\alpha r)^d$ after time $r$, and so the probability that a second mutation appears in this ball within $t$ time units after the original mutation occurs is
\begin{equation}\label{prob1}
1 - \exp \bigg(-\int_0^t \mu_2 \gamma_d (\alpha r)^d \: dr \bigg) = 1 - \exp \bigg(- \frac{\gamma_d}{d+1} \cdot \mu_2 \alpha^d t^{d+1}\bigg).
\end{equation}
It follows that the second mutation occurs when $t$ is comparable to $(\mu_2 \alpha^d)^{-1/(d+1)}$.  Therefore, when $(\mu_2 \alpha^d)^{-1/(d+1)} \ll 1/(N \mu_1)$, which is equivalent to our second assumption, this second mutation appears more quickly than the first mutation.  The dominant waiting time is therefore the time to wait for the first mutation, and we have, and we have
$$N \mu_1 \sigma_2 \Rightarrow W,\hspace{.3in}W \sim \textup{Exponential}(1).$$  
This result was proved by Durrett, Foo, and Leder, as part of Theorem 3 of \cite{dfl16}, so we do not give a proof here.  This case can be illustrated as follows:
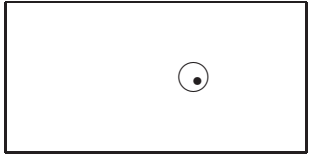
\begin{center}
\setlength{\unitlength}{1cm}
\begin{picture}(4,2)
\put(0,0){\line(1,0){4}}
\put(4,0){\line(0,1){2}}
\put(4,2){\line(-1,0){4}}
\put(0,2){\line(0,-1){2}}
\put(2.5,1){\circle{0.4}}
\put(2.55,0.95){\circle*{0.1}}
\end{picture}
\end{center}

\bigskip
\noindent {\bf Case 5}: ${\displaystyle \mu_1 \gg \frac{\alpha}{N^{(d+1)/d}}}$ and ${\displaystyle \frac{\mu_2 \alpha^d}{(N \mu_1)^{d+1}} \rightarrow c \in (0, \infty)}.$
\bigskip

In this case, the time between the first and second mutations is the same order of magnitude as the time to wait for the first mutation.  As a result, there could be several small regions with one mutation before the second mutation appears.  Writing $t' = t/(N \mu_1)$, using (\ref{prob1}), and making the substitution $y = N \mu_1 (t' - r)$, we have, as shown in Theorem 4 of \cite{dfl16},
\begin{align*}
P(N \mu_1 \sigma_2 > t) &\approx \exp \bigg( - \int_0^{t'} N \mu_1 \Big(1 - \exp \Big( - \frac{\gamma_d}{d+1} \cdot \mu_2 \alpha^d (t' - r)^{d+1} \Big) \Big) \: dr \bigg) \\
&\rightarrow \exp \bigg( - \int_0^t \Big( 1 - \exp \Big( - \frac{c \gamma_d y^{d+1}}{d+1} \Big) \Big) \: dy \bigg).
\end{align*}
Note that assumption (A1) in \cite{dfl16} is equivalent to the condition $\mu_1 \gg \alpha/N^{(d+1)/d}$ when $\mu_2 \alpha^d$ and $(N \mu_1)^{d+1}$ are the same order of magnitude, as we are assuming here.  Also, the assumption $\mu_2 \alpha^2/(N \mu_1)^{d+1} \rightarrow c$ is equivalent to $\Gamma \rightarrow 1/c$ in the notation of \cite{dfl16}.  This case can be illustrated as follows:
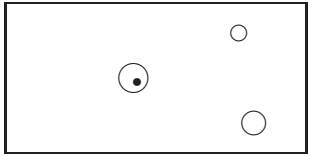
\begin{center}
\setlength{\unitlength}{1cm}
\begin{picture}(4,2)
\put(0,0){\line(1,0){4}}
\put(4,0){\line(0,1){2}}
\put(4,2){\line(-1,0){4}}
\put(0,2){\line(0,-1){2}}
\put(3.1, 1.6){\circle{0.2}}
\put(1.7,1){\circle{0.4}}
\put(1.75,0.95){\circle*{0.1}}
\put(3.3, 0.4){\circle{0.3}}
\end{picture}
\end{center}

\bigskip
\noindent {\bf Case 6}: ${\displaystyle \mu_1 \gg \frac{\alpha}{N^{(d+1)/d}}}$ and ${\displaystyle \frac{(\mu_1 \alpha^d)^{1/(d+1)}}{N} \ll \mu_2 \ll \frac{(N \mu_1)^{d+1}}{\alpha^d}}$.
\bigskip

The second inequality in the second assumption ensures that the second mutation will not appear until the number of regions with one mutation is large.  When the number of regions with one mutation is large, the fraction of the space filled with type 1 individuals should be well approximated by its expectation because no individual region contributes a large fraction of the type 1 individuals.  At time $t$, the probability that a particular site is occupied by an individual of type 1 or higher is
\begin{equation}\label{qteq}
q(t) = 1 - \exp \bigg( -\int_0^t \mu_1 \gamma_d (\alpha r)^d \: dr \bigg) = 1 - \exp \bigg( - \frac{\gamma_d}{d+1} \cdot \mu_1 \alpha^d t^{d+1} \bigg),
\end{equation}
which means that mutants fill a large fraction of the space when $t$ is of the order $(\mu_1 \alpha^d)^{-1/(d+1)}$.  For smaller values of $t$, we can use the approximation $1 - e^{-x} \approx x$ to estimate $q(t)$, and therefore
$$P(\sigma_2 > t) \approx \exp \bigg( - \int_0^t N \mu_2 q(r) \: dr \bigg) \approx \exp \bigg( - \frac{\gamma_d}{(d+1)(d+2)} \cdot N \mu_1 \mu_2 \alpha^d t^{d+2} \: dt \bigg).$$  This means that the second mutation arises when $t$ is of the order $(N \mu_1 \mu_2 \alpha^d)^{-1/(d+2)}$, which is much smaller than the time that it takes for mutants to fill a large fraction of the space precisely when the first inequality in the second assumption holds.  We then get
$$P\big( (N \mu_1 \mu_2 \alpha^d)^{1/(d+2)} \sigma_2 > t \big) \rightarrow \exp \bigg( - \frac{\gamma_d t^{d+2}}{(d+1)(d+2)} \bigg).$$  This result is Part 1 of Theorem \ref{thm:6and7andCaseII} below.  This case can be illustrated as follows:
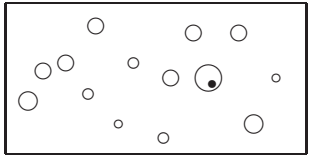
\begin{center}
\setlength{\unitlength}{1cm}
\begin{picture}(4,2)
\put(0,0){\line(1,0){4}}
\put(4,0){\line(0,1){2}}
\put(4,2){\line(-1,0){4}}
\put(0,2){\line(0,-1){2}}
\put(3.1, 1.6){\circle{0.2}}
\put(2.7,1){\circle{0.35}}
\put(2.75, 0.93){\circle*{0.1}}
\put(3.3, 0.4){\circle{0.25}}
\put(0.3, 0.7){\circle{0.25}}
\put(0.5, 1.1){\circle{0.2}}
\put(1.2, 1.7){\circle{0.2}}
\put(1.1, 0.8){\circle{0.15}}
\put(2.2, 1.0){\circle{0.2}}
\put(2.5, 1.6){\circle{0.2}}
\put(1.7, 1.2){\circle{0.15}}
\put(2.1, 0.2){\circle{0.15}}
\put(3.6, 1.0){\circle{0.1}}
\put(1.5, 0.4){\circle{0.1}}
\put(0.8, 1.2){\circle{0.2}}
\end{picture}
\end{center}

\bigskip
\noindent {\bf Case 7}: ${\displaystyle \mu_1 \gg \frac{\alpha}{N^{(d+1)/d}}}$ and ${\displaystyle \frac{N \mu_2}{(\mu_1 \alpha^d)^{1/(d+1)}} \rightarrow c \in (0, \infty)}.$
\bigskip

In this case, the time when the second mutation appears is the same order of magnitude as the time when type 1 individuals start to fill a large fraction of the space.  This means that the overlaps between different type 1 regions become significant and we can no longer use the approximation to $q(t)$ that was used in Case 6.  Instead, we use (\ref{qteq}) directly.  Writing $t' = t/(N \mu_1 \mu_2 \alpha^d)^{1/(d+2)}$ and making the substitution $y = r(N \mu_1 \mu_2 \alpha^d)^{1/(d+2)}$, we have
\begin{align*}
P \big( (N \mu_1 \mu_2 \alpha^d)^{1/(d+2)} \sigma_2 > t \big) &\approx \exp \bigg( - \int_0^{t'} N \mu_2 \Big(1 - \exp \Big( - \frac{\gamma_d}{d+1} \cdot \mu_1 \alpha^d r^{d+1} \Big) \Big) \: dr \bigg) \\
&\rightarrow \exp \bigg( - c^{(d+1)/(d+2)} \int_0^t \Big(1 - \exp \Big( -\frac{\gamma_d y^{d+1}}{(d+1) c^{(d+1)/(d+2)}} \Big) \Big) \: dy \bigg).
\end{align*}
This result is Part 2 of Theorem \ref{thm:6and7andCaseII} below.  This case can be illustrated as follows:
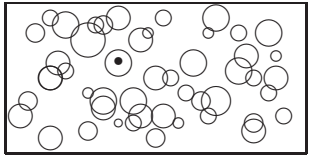
\begin{center}
\setlength{\unitlength}{1cm}
\begin{picture}(4,2)
\put(0,0){\line(1,0){4}}
\put(4,0){\line(0,1){2}}
\put(4,2){\line(-1,0){4}}
\put(0,2){\line(0,-1){2}}
\put(3.1, 1.6){\circle{0.2}}
\put(1.5,1.2){\circle{0.35}}
\put(1.5, 1.23){\circle*{0.1}}
\put(3.3, 0.4){\circle{0.25}}
\put(0.3, 0.7){\circle{0.25}}
\put(0.8, 1.1){\circle{0.2}}
\put(1.2, 1.7){\circle{0.2}}
\put(1.1, 0.8){\circle{0.15}}
\put(2.2, 1.0){\circle{0.2}}
\put(2.1, 1.8){\circle{0.2}}
\put(1.9, 1.6){\circle{0.15}}
\put(2.3, 0.4){\circle{0.15}}
\put(3.6, 1.3){\circle{0.15}}
\put(1.5, 0.4){\circle{0.1}}
\put(1.7, 0.4){\circle{0.2}}
\put(1.8, 0.5){\circle{0.3}}
\put(3.1, 1.1){\circle{0.35}}
\put(3.2, 1.3){\circle{0.3}}
\put(2.7, 0.5){\circle{0.2}}
\put(2.8, 1.8){\circle{0.35}}
\put(0.4, 1.6){\circle{0.25}}
\put(0.6, 1.8){\circle{0.2}}
\put(0.8, 1.7){\circle{0.35}}
\put(2,1){\circle{0.3}}
\put(2.5, 1.2){\circle{0.35}}
\put(2.7, 1.6){\circle{0.15}}
\put(0.6, 1.0){\circle{0.3}}
\put(3.5, 1.6){\circle{0.35}}
\put(3.7, 0.5){\circle{0.2}}
\put(1.8, 1.5){\circle{0.3}}
\put(2.1, 0.5){\circle{0.3}}
\put(1.1, 0.3){\circle{0.25}}
\put(0.6, 1.0){\circle{0.3}}
\put(1.7, 0.7){\circle{0.35}}
\put(2.4, 0.8){\circle{0.2}}
\put(0.7, 1.2){\circle{0.3}}
\put(3.3, 0.3){\circle{0.3}}
\put(2.6, 0.7){\circle{0.25}}
\put(0.2, 0.5){\circle{0.3}}
\put(3.3, 1.0){\circle{0.2}}
\put(2.0, 0.2){\circle{0.25}}
\put(0.6, 0.2){\circle{0.3}}
\put(1.5, 1.8){\circle{0.3}}
\put(1.3, 1.7){\circle{0.25}}
\put(3.6, 1.0){\circle{0.3}}
\put(3.5, 0.8){\circle{0.2}}
\put(2.8, 0.7){\circle{0.4}}
\put(1.3, 0.6){\circle{0.3}}
\put(1.3, 0.7){\circle{0.35}}
\put(1.1, 1.5){\circle{0.45}}
\end{picture}
\end{center}

\bigskip
\noindent {\bf Case 8}: ${\displaystyle \mu_1 \gg \frac{\alpha}{N^{(d+1)/d}}}$ and ${\displaystyle \mu_2 \ll \frac{(\mu_1 \alpha^d)^{1/(d+1)}}{N}}$.
\bigskip

In this case, the second mutation does not appear until after the space has been almost completely filled with many type $1$ regions.  By that point, second mutations are occurring at rate approximately $N \mu_2$, so just as in Case 2, we have
$$N \mu_2 \sigma_2 \Rightarrow W,\hspace{.3in}W \sim \textup{Exponential}(1).$$
See Theorem \ref{thm:sigma2Case8} below for the precise statement and proof.

\bigskip
\noindent {\bf Case 9}: ${\displaystyle \mu_1 \asymp \frac{\alpha}{N^{(d+1)/d}}}$ and ${\displaystyle \mu_2 \gg \frac{(N \mu_1)^{d+1}}{\alpha^d}}.$
\bigskip

When $\mu_1 \asymp \alpha/N^{(d+1)/d}$, the time required for the first mutation to appear is comparable to the time required for the first mutation to fixate.  When $\mu_2 \gg (N \mu_1)^{d+1}/\alpha^d$, the second mutation appears on a faster time scale.  Therefore, the waiting time for the first mutation is the dominant waiting time, which means $$N \mu_1 \sigma_2 \Rightarrow W,\hspace{.3in}W \sim \textup{Exponential}(1).$$
This was proved as part of Theorem 3 in \cite{dfl16}.  Note that this case is very similar to Case 4, and the two cases could easily be combined, as they were in \cite{dfl16}.

\bigskip
\noindent {\bf Case 10}: ${\displaystyle \mu_1 \asymp \frac{\alpha}{N^{(d+1)/d}}}$ and ${\displaystyle \mu_2 \ll \frac{(N \mu_1)^{d+1}}{\alpha^d}}.$
\bigskip

As in Case 9, the time required for the first mutation to appear is comparable to the time required for the first mutation to fixate.  However, when $\mu_2 \ll (N \mu_1)^{d+1}/\alpha^d$, it takes much longer for the second mutation to appear.  Therefore, the dominant waiting time is the time to wait for the second mutation after the first has fixated, and much as in Case 2, we have
$$N \mu_2 \sigma_2 \Rightarrow W,\hspace{.3in}W \sim \textup{Exponential}(1).$$ 
This result is proved in Theorem \ref{eq:Sig2Case10} below.

\bigskip
\noindent {\bf Case 11}: ${\displaystyle \mu_1 \asymp \frac{\alpha}{N^{(d+1)/d}}}$ and ${\displaystyle \mu_2 \asymp \frac{(N \mu_1)^{d+1}}{\alpha^d}}.$
\bigskip

In this case, the time that it takes for the first mutation to appear, the time that it takes for the first mutation to fixate once it has appeared, and the time that it takes for the second mutation to appear after the first one are all on the same time scale.  As a result, we can not consider just a single region of type 1 individuals as in Case 4, nor can we assume the type 1 regions are disjoint as in Case 5, nor can we assume that the fraction of the population with type~1 is approximately deterministic as in Cases 6 and 7, nor can we assume the type 1 individuals almost completely fill the space as in Case 8.  Instead, the full geometry of the problem must be taken into account.  We have $$N \mu_1 \sigma_2 \Rightarrow X,$$ where $X$ is a nondegenerate random variable.  This result is established in Corollary \ref{case11} below.  We do not have a complete description of the distribution of $X$, but some information about the distribution is established in Propositions \ref{distprop1} and \ref{distprop2} below.  This case can be illustrated as follows:
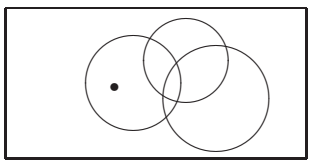
\begin{center}
\setlength{\unitlength}{1cm}
\begin{picture}(4,2)
\put(0,0){\line(1,0){4}}
\put(4,0){\line(0,1){2}}
\put(4,2){\line(-1,0){4}}
\put(0,2){\line(0,-1){2}}
\put(2.4, 1.3){\circle{1.1}}
\put(1.7,1){\circle{1.3}}
\put(1.45,0.95){\circle*{0.1}}
\put(2.8, 0.8){\circle{1.4}}
\end{picture}
\end{center}

\section{Waiting for $k$ mutations: results and heuristics}\label{ksec}

In this section, we summarize our asymptotic results for the distribution of $\sigma_k$ when $k \geq 3$.  We focus here on the case in which all of the mutation rates are the same, that is, we have $\mu_i = \mu$ for all $i$, although some of the results in Section \ref{proofsec} will be stated in greater generality when that can be done without additional effort.  This time, there are three cases to consider.

\bigskip
\noindent {\bf Case 1}: ${\displaystyle \mu \ll \frac{\alpha}{N^{(d+1)/d}}}$
\bigskip

In this case, it takes longer to wait for a mutation than it does for a mutation to fixate once it has appeared.  The waiting time for $k$ mutations is therefore approximately a sum of $k$ exponentially distributed random variables, which leads to the result
$$N \mu \sigma_k \Rightarrow Y, \hspace{.5in}Y \sim \textup{Gamma}(k,1).$$  This result is similar to Cases 1, 2, and 3 when $k = 2$ and can be deduced from Theorem \ref{expwait}, as noted in the paragraph following the statement of Theorem \ref{expwait}.

\bigskip
\noindent {\bf Case 2}: ${\displaystyle \mu \gg \frac{\alpha}{N^{(d+1)/d}}}$
\bigskip

In this case, mutations appear on a faster time scale than what is required for mutations to fixate, and so for $j \geq 2$, we end up with many small regions with $j-1$ mutations before any individual acquires a $j$th mutation.  Because there are many small regions with $j-1$ mutations, we are able to approximate the total size of these regions by its expectation.

We will define an approximation $v_k(t)$ to the volume of regions with at least $k$ mutations at time $t$.  We set $v_0(t) = N$ for all $t$.  Because, at time $r$, mutations to type $k$ are occurring at rate $\mu_k v_{k-1}(r)$, and such a mutation will lead to a type $k$ region of size $\gamma_d (\alpha(t - r))^d$ at time $t$, we define
\begin{equation}\label{vkpre}
v_k(t) = \int_0^t \mu_k v_{k-1}(r) \gamma_d (\alpha(t-r))^d \: dr.
\end{equation}
One can then verify by induction on $k$ that 
\begin{equation}\label{vkdef}
v_k(t) = \frac{\gamma_d^k (d!)^k}{(k(d+1))!} \bigg( \prod_{i=1}^k \mu_i \bigg) N \alpha^{kd} t^{k(d+1)}.
\end{equation}
To see this, one can first use the induction hypothesis to get $$v_k(t) = \frac{\gamma_d^k (d!)^{k-1}}{((k-1)(d+1))!} \bigg( \prod_{i=1}^k \mu_i \bigg) N \alpha^{kd} \int_0^t r^{(k-1)(d+1)} (t - r)^d \: dr,$$ and then make the substitution $y = r/t$ and use that $\int_0^1 y^a(1-y)^b \: dy = a!b!/(a+b+1)!$ for nonnegative integers $a$ and $b$ to obtain the result.  Equation (\ref{vkdef}) leads to the approximation
$$P(\sigma_k > t) \approx \exp \bigg(- \int_0^t \mu_k v_{k-1}(r) \: dr \bigg) = \exp \bigg( - \frac{\gamma_d^{k-1} (d!)^{k-1}}{((k-1)d + k)!} \bigg( \prod_{i=1}^k \mu_i \bigg) N \alpha^{(k-1)d} t^{(k-1)d + k} \bigg).$$  In particular, defining
\begin{equation}
\label{eq:beta_def}
\beta_k=\bigg(N\alpha^{(k-1)d}\prod_{i=1}^k\mu_i\bigg)^{-1/((k-1)d+k)},
\end{equation}
we get
\begin{equation}\label{limkbeta}
P \big( \sigma_k > \beta_k t) \rightarrow \exp \bigg( - \frac{\gamma_d^{k-1} (d!)^{k-1}}{((k-1)d + k)!} \cdot t^{(k-1)d + k} \bigg).
\end{equation}
This result is part 3 of Theorem \ref{thm:6and7andCaseII}.

In (\ref{vkpre}), we are double counting the volume in places where two or more type $k$ regions overlap.  Consequently, (\ref{vkdef}) will only be a good approximation to the total volume of the type $k$ regions if this overlap is small.  This will be the case if $v_k(t) \ll v_{k-1}(t)$, so that only a small fraction of the cells that have acquired at least $k-1$ mutations have also acquired a $k$th mutation.  We have $v_k(t) \ll v_{k-1}(t)$ if and only if $t \ll (\mu_k \alpha^d)^{-1/(d+1)}$.  Indeed, equation (\ref{limkbeta}) indicates that $\sigma_k$ should be of the order $\beta_k$, and when $\mu_i = \mu$ for all $i$, one can check that the condition $\beta_k \ll (\mu \alpha^d)^{-1/(d+1)}$ is equivalent to the condition for Case 2.

Note that if we use (\ref{limkbeta}) to obtain an asymptotic expression for $P(\sigma_k \leq \beta_k t)$ for small $t$, then use the approximation $1 - e^{-x} \approx x$ and differentiate with respect to $t$, we see that the rate of cancer incidence at time $t$ is roughly proportional to $t^{(k-1)(d+1)}$.  This is different from the power laws obtained by Armitage and Doll \cite{armdoll} in a non-spatial setting.

\bigskip
\noindent {\bf Case 3}: ${\displaystyle \mu \asymp \frac{\alpha}{N^{(d+1)/d}}}$
\bigskip

In this case the time at which the first mutation appears, the time at which the regions with one mutation spread to a significant fraction of the space, the time at which a second mutation appears, the time at which the regions with two mutations spread to a significant fraction of the space, and the time at which a third mutation appears are all the same order of magnitude.  That is, at the time the third mutation appears, there are already large regions with two mutations inside large regions with type 1 mutations, and these regions may overlap.  We have
\begin{equation}\label{3Xeq}
N \mu \sigma_k \Rightarrow X,
\end{equation}
where $X$ is a nondegenerate random variable.  This result is a special case of Theorem \ref{k3case3} below, and is similar to Case 11 when $k = 2$.  We are unable to describe completely the distribution of $X$, but Propositions \ref{distprop1} and \ref{distprop2} provide some information about the distribution.

\section{Proofs of Limit Theorems}\label{proofsec}

We first introduce some notation.  Denote the $d$-dimensional torus of side-length $L$ by ${\cal T} = [0, L]^d$.  For real numbers $x, y \in [0,L]$, define their distance as $$d_L(x,y) = \min\{|x - y|, L - |x - y|\}.$$  For points $x, y \in {\cal T}$, we write $x = (x^1, \dots, x^d)$ and $y = (y^1, \dots, y^d)$ and define their distance as $$|x - y|^2 = \sum_{i=1}^d d_L(x^i, y^i)^2.$$  Denote a ball of radius $r$ centered at $x$ by $B_x(r)$.  For a set $A \subset {\cal T}$, denote its Lebesgue measure by $|A|$.  Let $N = L^d = |{\cal T}|$ be the volume of the torus.

For each $x \in {\cal T}$ and $t \geq 0$, denote the type of that space-time location by $T(x,t)$.  For $i \in \N_0$, define the set of type $i$ sites by $$\chi_i(t) = \{x \in {\cal T}: T(x,t) = i\}$$ and the set of sites whose type is greater than or equal to $i$ by $$\psi_i(t) = \{x \in {\cal T}: T(x,t) \geq i\}.$$  Let $X_i(t) = |\chi_i(t)|$ denote the total volume of type $i$ sites at time $t$, and let $Y_i(t) = |\psi_i(t)|$ denote the total volume of sites at time $t$ whose type is greater than or equal to $i$.

It will be useful to construct our whole process from a sequence of independent Poisson point processes.  Let $(\Pi_k)_{k=1}^{\infty}$ be a sequence of independent Poisson point processes on ${\cal T} \times [0, \infty)$ such that $\Pi_k$ has constant intensity $\mu_k$.  The points of $\Pi_k$ represent the space-time points at which an individual can acquire a $k$th mutation.  More specifically, if $(x, t)$ is a point of $\Pi_k$ and $x \in \chi_{k-1}(t)$, then we say that the individual at site $x$ mutates to type $k$ at time $t$.  The type $k$ individuals then spread outward at rate $\alpha$.  For example, if
$\Pi_1 \cap ({\cal T} \times [0, t]) = \{(x_1, t_1), \dots (x_k, t_k)\}$, we have $$\psi_1(t) = \bigcup_{j=1}^k B_{x_j}(\alpha(t - t_j)).$$  Let
\begin{equation}\label{Rxtdef}
R(x,t) = \{(y, s) \in {\cal T} \times [0, t]: x \in B_y(\alpha(t-s))\}.
\end{equation}
Note that if a mutation occurs at a space-time location $(y, s) \in R(x,t)$, then the mutation will spread to the site $x$ by time $t$.  Therefore, we have $x \in \psi_1(t)$ if and only if $\Pi_1 \cap R(x,t) \neq \emptyset$.  More generally, for $k \geq 1$, we have $x \in \psi_k(t)$ if and only if there is a point $(y, s)$ of $\Pi_k$ such that $(y, s) \in R(x,t)$ and $y \in \psi_{k-1}(s)$.  Note that this claim would still hold if we replaced the condition $y \in \psi_{k-1}(s)$ by the condition $y \in \chi_{k-1}(s)$.  However, it will be more convenient to work with $\psi_{k-1}(s)$ because with this construction, for all $k \geq 2$, the random set $\psi_{k-1}(s)$ is completely determined by the Poisson processes $\Pi_1, \dots, \Pi_{k-1}$.

\subsection{Cases 1, 2, 3, and 10 for $k=2$, and Case 1 for $k \geq 3$}

In this subsection, we establish the results in the cases when it takes longer for mutations of a given type to appear than it does for them to fixate, in which case the time to wait for $k$ mutations is well approximated by a sum of independent exponentially distributed waiting times.  We will prove the following theorem, which includes Cases 1, 2, and 3 when $k = 2$ and Case 1 when $k = 3$.

\begin{Theo}\label{expwait}
Suppose $\mu_i \ll \alpha/ N^{(d+1)/d}$ for $i \in \{1, \dots, k-1\}$.  Suppose there exists $j \in \{1, \dots, k\}$ such that $\mu_j \ll \alpha/N^{(d-1)/d}$ and $$\frac{\mu_i}{\mu_j} \rightarrow c_i \in (0, \infty] \hspace{.5in}\mbox{for all } i \in \{1, \dots, k\}.$$  Let $W_1, \dots, W_k$ are independent random variables such that $W_i$ has an exponential distribution with rate parameter $c_i$ if $c_i < \infty$ and $W_i = 0$ if $c_i = \infty$.  Then $$N \mu_j \sigma_k \Rightarrow W_1 + \dots + W_k.$$ 
\end{Theo}

Note that if $k = 2$ and the conditions of Case 1 are satisfied, then we take $j = 1$ and get $c_1 = 1$ and $c_2 = \infty$.  This leads to the result that $N \mu_1 \sigma_2 \Rightarrow W_1$, which is the result of Case 1.  If $k = 2$ and the conditions of Case 2 are satisfied, then we take $j = 2$ and see that $c_1 = \infty$ and $c_2 = 1$.  It follows that $N \mu_2 \sigma_2 \Rightarrow W_2$, which is the result for Case 2.  If $k = 2$ and the conditions of Case 3 are satisfied, then we take $j = 1$, which implies that $c_1 = 1$ and $c_2 = c$.  Then we have $N \mu_1 \sigma_2 \Rightarrow W_1 + W_2$, matching the result for Case 3.  Finally, suppose $k \geq 3$ and the conditions of Case 1 are satisfied, so that $\mu_i = \mu$ for all $i$.  Then, for any choice of $j$, we have $c_i = 1$ for all $i \in \{1, \dots, k\}$.  This leads to the result that $N \mu \sigma_k \Rightarrow W_1 + \dots + W_k$, where each $W_i$ has an exponential distribution with rate $1$ and thus $W_1 + \dots + W_k$ has a Gamma$(k,1)$ distribution, confirming the result that we previously claimed.

\begin{proof}
Set $t_0 = 0$, and for $i \geq 1$, let
\begin{equation}\label{tidef}
t_i = \inf\{t > 0: Y_i(t) = N\}
\end{equation}
be the first time at which all individuals have type $i$ or higher.  Define the time elapsed between the first appearance of a type $i$ individual and the time at which all individuals have type $i$ or higher as
\begin{equation}\label{tihatdef}
\hat{t}_i=t_i-\sigma_i.
\end{equation}
For $i \geq 2$, let $A_i$ be the event that $\Pi_i \cap ({\cal T} \times [\sigma_{i-1}, t_{i-1}]) = \emptyset$.  On the event $A_i$, no individual acquires an $i$th mutation before the entire population has type $i-1$.  For $i \geq 1$, let $$\hat{\sigma}_i = \inf\{t: \Pi_i \cap ({\cal T} \times [t_{i-1}, t]) \neq \emptyset\},$$ which is the first time, after time $t_{i-1}$, that there is a potential mutation to type $i$.  We have $\sigma_i = \hat{\sigma}_i$ on the event $A_i$.  In particular, on the event $A_2 \cap \dots \cap A_k$, we have
\begin{equation}\label{sigmadecompose}
\sigma_k = \sum_{i=1}^k (\hat{\sigma}_i - t_{i-1}) + \sum_{i=1}^{k-1} \hat{t}_i.
\end{equation}
Here ${\hat t}_i$ is the time required for the $i$th mutation to spread to the entire population once it appears, and $\hat{\sigma}_i - t_{i-1}$ is the waiting time for the $i$th mutation to appear once $i-1$ mutations have fixated in the population.

For $x, y \in {\cal T}$, we have $|x - y| \leq \frac{\sqrt{d}}{2}L$, which implies that
\begin{equation}\label{hatbound}
\hat{t}_i \leq \frac{\sqrt{d}N^{1/d}}{2 \alpha}.
\end{equation}
In particular, because $\mu_j \ll \alpha/N^{(d+1)/d}$, we have
\begin{equation}\label{gapstozero}
N \mu_j \sum_{i=1}^{k-1} \hat{t}_i \rightarrow 0.
\end{equation}
Moreover, because the Poisson process $\Pi_i$ has constant rate $\mu_i$, and the random time $t_{i-1}$ depends only on $\Pi_1, \dots, \Pi_{i-1}$ and thus is independent of $\Pi_i$, the times $\hat{\sigma}_1 - t_0$, $\hat{\sigma}_2 - t_1$, \dots, $\hat{\sigma}_k - t_{k-1}$ are independent, and $\hat{\sigma}_i - t_{i-1}$ has an exponential distribution with rate $N \mu_i$.  Because $\mu_i/\mu_j \rightarrow c_i$ by assumption, it follows that $N \mu_j (\hat{\sigma}_i - t_{i-1}) \Rightarrow W_i$ for all $i \in \{1, \dots, k\}$.  Combining this result with (\ref{gapstozero}) and the independence of the random variables $\hat{\sigma}_i - t_{i-1}$, we have $$N \mu_j \left( \sum_{i=1}^k (\hat{\sigma}_i - t_{i-1}) + \sum_{i=1}^{k-1} \hat{t}_i \right) \: \Rightarrow W_1 + \dots + W_k.$$

In view of (\ref{sigmadecompose}), the statement of the theorem will follow if we show that $P(A_2 \cap \dots \cap A_k) \rightarrow 1$.  Because the Poisson point process $\Pi_i$ has constant rate $\mu_i$, it follows from (\ref{hatbound}) that $$P(A_i) \geq \exp \bigg(- N \mu_i \cdot \frac{ \sqrt{d} N^{1/d}}{2 \alpha} \bigg) \geq 1 - \frac{\sqrt{d} \mu_i N^{(d+1)/d}}{2 \alpha}.$$  Because $\mu_i \ll \alpha/N^{(d+1)/d}$ for $i \in \{1, \dots, k-1\}$ by assumption, it follows that $P(A_i) \rightarrow 1$ for $i \in \{2, \dots, k-1\}$.  If we also have $\mu_k \ll \alpha/N^{(d+1)/d}$, then $P(A_k) \rightarrow 1$ as well, and the proof is complete.

On the other hand, suppose we do not have $\mu_k \ll \alpha/N^{(d+1)/d}$.  Then the argument needs to be adjusted because the $k$th and final mutation may happen faster than the others, and in particular may occur between times $\sigma_{k-1}$ and $t_{k-1}$.  In this case, we have $j \neq k$ and $\mu_k/\mu_j \rightarrow \infty$, which means $W_k = 0$.  In place of (\ref{sigmadecompose}), on the event $A_2 \cap \dots \cap A_{k-1}$, we can write
$$\sigma_k = (\sigma_k - \sigma_{k-1}) + \sum_{i=1}^{k-1} (\hat{\sigma}_i - t_{i-1}) + \sum_{i=1}^{k-2} \hat{t}_i,$$
and we know from the argument given above that $$N \mu_j \left( \sum_{i=1}^{k-1} (\hat{\sigma}_i - t_{i-1}) + \sum_{i=1}^{k-2} \hat{t}_i \right) \: \Rightarrow W_1 + \dots + W_{k-1}.$$  Furthermore, we still have $$\sigma_k - \sigma_{k-1} \leq \hat{t}_{k-1} + (\hat{\sigma}_k - t_{k-1}).$$  We have $N \mu_j \hat{t}_{k-1} \rightarrow 0$, and because $N \mu_j (\hat{\sigma}_k - t_{k-1})$ has an exponential distribution with rate $\mu_k/\mu_j \rightarrow \infty$, we have $N \mu_j (\hat{\sigma}_k - t_{k-1}) \rightarrow 0$ in probability, which now implies the conclusion of the theorem.
\end{proof}

The following theorem establishes the result for Case 10 when $k = 2$.

\begin{Theo}\label{eq:Sig2Case10}
Assume that the assumptions for Case 10 hold. Then for $t>0$,
$$
\lim_{N\to\infty}P(N\mu_2\sigma_2>t)=e^{-t}.
$$
\end{Theo}

\begin{proof}
Define $t_1$ and $\hat{t}_1$ as in (\ref{tidef}) and (\ref{tihatdef}).  Write
$$
N\mu_2\sigma_2=N\mu_2t_1+N\mu_2(\sigma_2-t_1),
$$
and recall that $t_1=\sigma_1+\hat{t}_1$. We first establish that as $N\to\infty$, $N\mu_2t_1\to 0$.  Because
$$
\frac{\mu_2}{\mu_1}\ll N\left(\frac{N\mu_1}{\alpha}\right)^d\asymp 1,
$$
we have $\mu_2 \ll \mu_1$, and therefore $N\mu_2\sigma_1\to 0$.  Also, under the assumptions of Case 10,
$$
N \mu_2 \hat{t}_1 \leq \frac{\sqrt{d} N^{(d+1)/d} \mu_2}{2 \alpha} \asymp\frac{\mu_2}{\mu_1}\ll 1.
$$
We can thus conclude that as $N\to\infty$, $N\mu_2t_1\to 0$.

It thus remains to find the limit of $P(N\mu_2(\sigma_2-t_1)>t).$
Note that because $Y_1(s) = N$ for $s \geq t_1$, we have
\begin{align*}
P\left(N\mu_2(\sigma_2-t_1)>t\right)&=E\left[\exp\left(-\mu_2\int_0^{t_1}Y_1(s)ds-\mu_2\int_{t_1}^{t_1+t/(N\mu_2)}Y_1(s)ds\right)\right] \\
&=
e^{-t}E\left[ \exp\left(-\mu_2\int_0^{t_1}Y_1(s)ds\right)\right] .
\end{align*}
Since $0\leq\mu_2\int_0^{t_1}Y_1(s)ds\leq N \mu_2 t_1,$ we know that as $N\to\infty$, $\mu_2\int_0^{t_1}Y_1(s)ds\to 0$.  Thus, by the dominated convergence theorem,
$$
\lim_{N\to\infty}E\left[\exp\left(-\mu_2\int_0^{t_1}Y_1(s)ds\right)\right] = 1.
$$
The result follows.
\end{proof}

\subsection{Proof for Cases 6, 7 and 8 when $k=2$ and Case 2 when $k\geq 3$}

We begin with a simple first moment result.

\begin{Lemma}\label{meanY}
If $0 \leq t \leq N^{1/d}/(2 \alpha)$, then $$E[Y_1(t)] = N \bigg( 1 - \exp \bigg( - \frac{\mu_1 \gamma_d \alpha^d t^{d+1}}{d+1} \bigg) \bigg).$$
\end{Lemma}

\begin{proof}
Recall that $0 \in \psi_1(t)$ if and only if $\Pi_1 \cap R(0,t) \neq \emptyset$, where $R(0,t)$ was defined in (\ref{Rxtdef}). Using also the spatial homogeneity of the torus, we have
$$E[Y_1(t)] = E \bigg[ \int_{{\cal T}} \1_{\{x \in \psi_1(t)\}} \: dx \bigg] = N P(0 \in \psi_1(t)) = N(1 - e^{-\mu_1 |R(0,t)|}).$$
When $0 \leq t \leq N^{1/d}/(2 \alpha)$, we have $\alpha t \leq L/2$, and therefore a ball of radius $\alpha t$ in the torus of side length $L$ has the same volume as a ball of the same radius in $\R^d$.  Therefore,
$$|R(0,t)| = \int_0^t |B_0(\alpha(t-s))| \: ds = \int_0^t \gamma_d \alpha^d (t-s)^d \: ds = \frac{\gamma_d \alpha^d t^{d+1}}{d+1},$$ and the result follows.  
\end{proof}

When $k \geq 2$, we are not able to obtain an exact formula for $E[Y_k(t)]$.  However, we are able to obtain useful estimates.  Define
$$y_k(t) = \frac{v_k(t)}{N} = \frac{\gamma_d^k (d!)^k \alpha^{kd}}{(k(d+1))!} \bigg( \prod_{i=1}^k \mu_i \bigg) t^{k(d+1)}.$$
The next lemma gives an upper bound for $E[Y_k(t)]$.

\begin{Lemma}\label{lemma:Occupancy}
For $t>0$ and integers $k \geq 0$, we have $P(0\in \psi_k(t))\leq y_k(t)$ and $E[Y_k(t)]\leq v_k(t).$
\end{Lemma}

\begin{proof}
Recall that $0 \in \psi_k(t)$ if and only if there is a point of the Poisson process $\Pi_k$ in the set $R(0,t) \cap \{(x,s): x \in \psi_{k-1}(s)\}$.  For $k \geq 1$, define
\begin{equation}\label{Lambdadef}
\Lambda_{k-1}(t)=\mu_k \int \int_{R(0,t)}\1_{\{x \in \psi_{k-1}(s)\}} \: dx \: ds.
\end{equation}
Conditional on $\Pi_1, \dots, \Pi_{k-1}$, the number of points of $\Pi_k$ in the set $R(0,t) \cap \{(x,s): x \in \psi_{k-1}(s)\}$ has the Poisson distribution with mean $\Lambda_{k-1}(t)$.  Note that for this claim to hold, it is important to work with $\psi_{k-1}(s)$ rather than $\chi_{k-1}(s)$ because $\psi_{k-1}(s)$ depends only on $\Pi_1, \dots, \Pi_{k-1}$ and therefore is independent of $\Pi_k$.  Therefore, using the spatial homogeneity of the process,
\begin{align}
\label{eq:OccupancyIneq}
P\left(0\in \psi_k(t)\right) &= E\left[1-e^{-\Lambda_{k-1}(t)}\right] \nonumber \\
&\leq E\left[\Lambda_{k-1}(t)\right]\nonumber\\
&= \mu_k \int_0^t \int_{B_0(\alpha(t-s))}P\left(x\in\psi_{k-1}(s)\right) \: dx \:ds\nonumber\\
&= \mu_k \int_0^t P(0\in\psi_{k-1}(s)) |B_0(\alpha(t-s))| \: ds\nonumber\\
&= \mu_k \gamma_d \alpha^d \int_0^t P(0\in\psi_{k-1}(s)) (t - s)^d \: ds.
\end{align}

We now prove the upper bound on $P(0 \in \psi_k(t))$ by induction.  Because $y_0(t) = 1$ for all $t$, the $k = 0$ case is trivial.  Let $k \geq 1$, and suppose $P(0 \in \psi_{k-1}(t)) \leq y_{k-1}(t)$ for all $t > 0$.  By (\ref{eq:OccupancyIneq}), for all $t > 0$,
\begin{align*}
P(0 \in \psi_k(t)) &\leq \mu_k \gamma_d \alpha^d \int_0^t \frac{\gamma_d^{k-1} (d!)^{k-1} \alpha^{(k-1)d}}{((k-1)(d+1))!} \bigg( \prod_{i=1}^{k-1} \mu_i \bigg) s^{(k-1)(d+1)} (t-s)^d \: ds \\
&= \bigg( \prod_{i=1}^k \mu_i \bigg) \frac{\gamma_d^k (d!)^{k-1} \alpha^{kd}}{((k-1)(d+1))!} \int_0^t s^{(k-1)(d+1)}(t - s)^d \: ds.
\end{align*}
Making the substitution $u = s/t$ and then using that $\int_0^1 u^a (1-u)^b \: du = a!b!/(a+b+1)!$ for nonnegative integers $a$ and $b$, we get
\begin{equation}\label{indupper}
P(0 \in \psi_k(t)) \leq \bigg( \prod_{i=1}^k \mu_i \bigg) \frac{\gamma_d^k (d!)^{k-1} \alpha^{kd}}{((k-1)(d+1))!} \cdot t^{k(d+1)} \int_0^1 u^{(k-1)(d+1)} (1 - u)^d \: du = y_k(t).
\end{equation}
Thus, by induction, $P(0\in \psi_k(t))\leq y_k(t)$ for all $t > 0$ and all nonnegative integers $k$.
The upper bound for $E[Y_k(t)]$ now follows from the formula
\begin{equation}\label{YfromPsi}
E[Y_k(t)]=\int_{\mathcal{T}}\1_{\{x\in\psi_k(t)\}} \: dx
\end{equation}
and the spatial homogeneity of $\mathcal{T}$.
\end{proof}

When $\mu_j \alpha^d t^{d+1} \rightarrow 0$ for all $j \in \{1, \dots, k\}$, we see that $y_j(t) \rightarrow 0$ for all $j \in \{1, \dots, k\}$, so most individuals have not yet acquired mutations by time $t$.  Also, we have $y_j(t)/y_{j-1}(t) \rightarrow 0$ for all $j \in \{1, \dots, k\}$, which means that among the individuals with at least $j-1$ mutations, only a small fraction will have acquired a $j$th mutation.  As a result, there will not be much overlap in the regions affected by different mutations to type $j-1$.  Under this condition, $E[Y_k(t)]$ can be approximated by $v_k(t)$, as shown below.

\begin{Lemma}\label{thm:no_overlaps}
Fix a positive integer $k$.  Suppose $\mu_j \alpha^d t^{d+1} \rightarrow 0$ as $N \rightarrow \infty$ for all $j \in \{1, \dots, k\}$.  Then $$\lim_{N \rightarrow \infty} \frac{E[Y_k(t)]}{v_k(t)} = 1.$$
\end{Lemma}

\begin{proof}
Recall the definition of $\Lambda_{k-1}(t)$ from (\ref{Lambdadef}).
Using that $1 - e^{-x} \geq x - x^2/2$ for $x \geq 0$ and proceeding as in (\ref{eq:OccupancyIneq}), we get
$$P(0 \in \psi_k(t)) = E\left[1-e^{-\Lambda_{k-1}(t)}\right] \geq E[\Lambda_{k-1}(t)] - \frac{1}{2} E\left[\Lambda_{k-1}(t)^2\right].$$
We have
\begin{align*}
E\left[\Lambda_{k-1}(t)^2\right] &= \mu_k^2 E \left[ \bigg( \int_0^t \int_{B_0(\alpha(t - s))} \1_{\{x \in \psi_{k-1}(s)\}} \: dx \: ds \bigg)^2 \right] \\
&=\mu_k^2\int_0^{t}\int_{B_0(\alpha(t-s))}\int_0^{t}\int_{B_0(\alpha(t-r))}P(y\in\psi_{k-1}(s),x\in\psi_{k-1}(r))\: dx \: dr \: dy \: ds.
\end{align*}
Whether or not $y \in \psi_{k-1}(s)$ is determined entirely by the restrictions of the Poisson point processes $\Pi_1, \dots, \Pi_{k-1}$ to the space-time region $R(y, s)$, and likewise for the event $x \in \psi_{k-1}(r)$.  Note that
$$
R(x,r)\cap R(y,s)=\emptyset \Leftrightarrow  |x-y|>\alpha(s+r)\Leftrightarrow x\notin B_y(\alpha(s+r)), 
$$
and thus if $x\notin B_y(\alpha(s+r))$ we have
$$
P\left(y\in \psi_{k-1}(s),x\in\psi_{k-1}(r)\right)=P\left(y\in \psi_{k-1}(s)\right)P\left(x\in\psi_{k-1}(r)\right).
$$
Let $A(s,r,y) = B_0(\alpha(t-r)) \cap B_y(\alpha(s+r))^c$ and $B(s,r,y) = B_0(\alpha(t-r)) \cap B_y(\alpha(s+r))$.  Then
\begin{align*}
E\left[\Lambda_{k-1}(t)^2\right] &= \mu_k^2\int_0^{t}\int_{B_0(\alpha(t-s))}\int_0^{t}\int_{A(s,r,y)}P(y\in\psi_{k-1}(s))P(x\in\psi_{k-1}(r)) \: dx \: dr \: dy \: ds\\
&\hspace{.3in}+ \mu_k^2\int_0^{t}\int_{B_0(\alpha(t-s))}\int_0^{t}\int_{B(s,r,y)}P(y\in\psi_{k-1}(s),x\in\psi_{k-1}(r))\: dx \: dr \: dy \: ds\\
&\leq E[\Lambda_{k-1}(t)]^2+\mu_k^2\int_0^{t}\int_{B_0(\alpha(t-s))}P(y\in\psi_{k-1}(s)) \bigg( \int_0^{t}|B_0(\alpha(t-r))| \: dr \bigg) \: dy \: ds\\
&= E[\Lambda_{k-1}(t)]^2+\mu_k E[\Lambda_{k-1}(t)]  \bigg( \int_0^{t}|B_0(\alpha(t-r))| \: dr \bigg) \\
&= E[\Lambda_{k-1}(t)] \left( E[\Lambda_{k-1}(t)] + \frac{\mu_k \gamma_d \alpha^d t^{d+1}}{d+1} \right) \\
&\leq E[\Lambda_{k-1}(t)] \left( y_k(t) + \frac{\mu_k \gamma_d \alpha^d t^{d+1}}{d+1} \right).
\end{align*}
Let $\eps > 0$.  Because $\mu_k \alpha^d t^{d+1} \rightarrow \infty$ as $N \rightarrow \infty$, it follows that for sufficiently large $N$, we have
$$P(0 \in \psi_k(t)) \geq E[\Lambda_{k-1}(t)] \left(1 - \frac{y_k(t)}{2} - \frac{\mu_k \gamma_d \alpha^d t^{d+1}}{2(d+1)} \right) \geq (1 - \eps) E[\Lambda_{k-1}(t)].$$

From (\ref{eq:OccupancyIneq}), we have
$$E[\Lambda_{k-1}(t)] = \mu_k \gamma_d \alpha^d \int_0^t P(0 \in \psi_{k-1}(s)) (t-s)^d \: ds.$$  We now show by induction that for all $j \in \{0, 1, \dots, k\}$ and all $s \in [0, t]$, we have $$P(0 \in \psi_k(s)) \geq (1 - \eps)^k y_k(s).$$  Because $y_0(s) = 1$ for all $s \in [0, t]$ and $P(0 \in \psi_0(s)) = 1$, the result holds for $j = 0$.  Let $j \geq 1$, and suppose $P(0 \in \psi_{j-1}(s)) \geq (1 - \eps)^{j-1} y_{j-1}(s)$ for all $s \in [0, t]$.  Then for $s \in [0, t]$, using the induction hypothesis and repeating the calculation in the derivation of (\ref{indupper}),
\begin{align*}
P(0 \in \psi_j(s)) &\geq (1 - \eps) E[\Lambda_{j-1}(s)] \\
&\geq (1 - \eps) \mu_j \gamma_d \alpha^d \int_0^s P(0 \in \psi_{j-1}(u)) (s - u)^d \: du \\
&\geq (1 - \eps)^j \mu_j \gamma_d \alpha^d \int_0^s y_{j-1}(u) (s - u)^d \: du \\
&= (1 - \eps)^j y_j(s).
\end{align*}
Because $\eps > 0$ was arbitrary, the result now follows from (\ref{YfromPsi}), the spatial homogeneity of ${\cal T}$, and the upper bound in Lemma \ref{lemma:Occupancy}.
\end{proof}

We next establish a variance bound using the independence of points that have disjoint space-time cones.

\begin{Lemma}\label{varY}
For all $t \geq 0$, and positive integer $k$, we have $$\Var(Y_k(t)) \leq \gamma_d (2 \alpha t)^d E[Y_k(t)].$$
\end{Lemma}

\begin{proof}
We have
$$
E[Y_k(t)^2] = E \bigg[ \int_{{\cal T}} \int_{{\cal T}} \1_{\{x \in \psi_k(t)\}} \1_{\{y \in \psi_k(t)\}} \: dx\: dy \bigg] = \int_{{\cal T}} \int_{{\cal T}} P \big( x \in \psi_k(t), y \in \psi_k(t)\big) \: dx \: dy.
$$  
Note that if $R(x,t) \cap R(y,t) = \emptyset$ or, equivalently, if $x \notin B_y(2 \alpha t)$, then the events $\{x \in \psi_k(t)\}$ and $\{y \in \psi_k(t)\}$ are independent, and therefore $$P\big(x \in \psi_k(t), y \in \psi_k(t)\big) = P(x \in \psi_k(t))P(y \in \psi_k(t)).$$  On the other hand, if $x \in B_y(2 \alpha t)$, then $$P\big(x \in \psi_k(t), y \in \psi_k(t)\big) \leq P(y \in \psi_k(t)).$$  Because the volume of a ball in the torus is bounded above by the volume of a ball of the same radius in $\R^d$, it follows that
\begin{align*}
E[Y_k(t)^2] &\leq \int_{{\cal T}} \int_{{\cal T} \setminus B_y(2 \alpha t)} P(x \in \psi_k(t)) P(y \in \psi_k(t)) \: dx \: dy + \int_{{\cal T}} \int_{B_y(2 \alpha t)} P(y \in \psi_k(t)) \: dx \: dy. \\
&\leq \int_{{\cal T}} \int_{{\cal T}} P(x \in \psi_k(t)) P(y \in \psi_k(t)) \: dx \: dy + \gamma_d (2 \alpha t)^d \int_{{\cal T}} P(y \in \psi_k(t)) \: dy \\
&= (E[Y_k(t)])^2 + \gamma_d (2 \alpha t)^d E[Y_k(t)],
\end{align*}
which implies the result.
\end{proof}

The next result gives conditions under which the value of $Y_1(t)$ can be approximated by its expectation.  The condition $N \mu_1 t \rightarrow \infty$ ensures that many mutations have occurred by time $t$, which means the region $\psi_1(t)$ will not be dominated by the effect of a single mutation.  The condition $\alpha t \ll N^{1/d}$ ensures that no single mutation has had a chance to spread to a large fraction of the space by time $t$.  These conditions together stipulate that $\psi_1(t)$ consists of a union of many small balls, which are possibly overlapping.

\begin{Lemma}\label{conc1}
Suppose $N \mu_1 t \rightarrow \infty$ and $\alpha t \ll N^{1/d}$.  Then for all $\eps > 0$,
$$\lim_{N \rightarrow \infty} P\big((1 - \eps)E[Y_1(t)] \leq Y_1(t) \leq (1 + \eps)E[Y_1(t)] \big) = 1.$$
\end{Lemma}

\begin{proof}
By Lemma \ref{varY} and Chebyshev's Inequality,
\begin{equation}\label{var1cheb}
P \big(\big|Y_1(t) - E[Y_1(t)] \big| > \eps E[Y_1(t)]\big) \leq \frac{\Var(Y_1(t))}{\eps^2 (E[Y_1(t)])^2} = \frac{\gamma_d (2 \alpha t)^d}{\eps^2 E[Y_1(t)]}. 
\end{equation}
It remains to show that the right-hand side of (\ref{var1cheb}) tends to zero as $N \rightarrow \infty$.  Note that the assumption that $\alpha t \ll N^{1/d}$ means that the conclusion of Lemma \ref{meanY} holds for sufficiently large $N$.  We consider two cases.  First, suppose $\mu_1 \alpha^d t^{d+1} \rightarrow 0$.  Then by Lemma \ref{meanY}, as $N \rightarrow \infty$, we have $$E[Y_1(t)] \sim \frac{N \mu_1 \gamma_d \alpha^d t^{d+1}}{d+1}.$$ Therefore, the assumption that $N \mu_1 t \rightarrow \infty$ implies that $$\frac{\gamma_d (2 \alpha t)^d}{\eps^2 E[Y_1(t)]} \sim \frac{2^d(d+1)}{\eps^2} \cdot \frac{1}{N \mu_1 t} \rightarrow 0.$$  Alternatively, suppose $$\liminf_{N \rightarrow \infty} \mu_1 \alpha^d t^{d+1} > 0.$$  Then by Lemma \ref{meanY}, the expectation $E[Y_1(t)]$ is bounded below by a constant multiple of $N$, and therefore the assumption that $\alpha t \ll N^{1/d}$ implies that the right-hand side of (\ref{var1cheb}) tends to zero.  Because the right-hand side of (\ref{var1cheb}) tends to zero in both cases, a subsequence argument completes the proof.
\end{proof}

The next result is similar to Lemma \ref{conc1} but holds for $k \geq 2$.  Note that condition (\ref{largemucond}) below reduces to the condition that $N \mu_1 t \rightarrow \infty$ when $k = 1$.  This condition ensures that many mutations to type $k$ will happen before time $t$, which is necessary to obtain a concentration result.  The condition $\mu_j \alpha^d t^{d+1} \rightarrow 0$ is stronger than the corresponding hypothesis in Lemma \ref{conc1}.   As noted above, this condition ensures that among the individuals with at least $j-1$ mutations, only a small fraction will have acquired a $j$th mutation.

\begin{Lemma}\label{conck}
Suppose that as $N \rightarrow \infty$, we have $\mu_j \alpha^d t^{d+1} \rightarrow 0$ for all $j \in \{1, \dots, k\}$ and
\begin{equation}\label{largemucond}
\bigg( \prod_{i=1}^k \mu_i \bigg) N \alpha^{(k-1)d} t^{(k-1)d + k} \rightarrow \infty.
\end{equation}
Then
$$\lim_{N \rightarrow \infty} P\big((1 - \eps)E[Y_k(t)] \leq Y_k(t) \leq (1 + \eps)E[Y_k(t)] \big) = 1.$$
\end{Lemma}

\begin{proof}
By Lemma \ref{varY} and Chebyshev's Inequality,
\begin{equation}\label{varkcheb}
P \big(\big|Y_k(t) - E[Y_k(t)] \big| > \eps E[Y_k(t)]\big) \leq \frac{\Var(Y_k(t))}{\eps^2 (E[Y_k(t)])^2} = \frac{\gamma_d (2 \alpha t)^d}{\eps^2 E[Y_k(t)]}, 
\end{equation}
so it remains to show that $(\alpha t)^d/E[Y_k(t)] \rightarrow 0$.  By Lemma \ref{thm:no_overlaps}, this is equivalent to the condition that $(\alpha t)^d \ll v_k(t)$, which is equivalent to (\ref{largemucond}).
\end{proof}

Our next result establishes conditions when a monotone stochastic process can be well approximated by a deterministic function. 

\begin{Lemma}\label{conc:MonotoneProcess}
Suppose, for all positive integers $N$, $(Y_N(t), t \geq 0)$ is a nondecreasing stochastic process with finite mean for all $t>0$.  Assume there exist sequences of positive numbers $(\nu_N)_{N=1}^{\infty}$ and $(s_N)_{N=1}^{\infty}$ and a continuous nondecreasing function $g$ such that for all $t>0$ and $\eps>0$, we have
\begin{align}
\label{eq:MonotoneProcessCond1}
\lim_{N\to\infty} P\big(|Y_N(s_Nt)-E[Y_N(s_Nt)]|>\eps E[Y_N(s_Nt)]\big)=0
\end{align}
and
\begin{align}
\label{eq:MonotoneProcessCond2}
\lim_{N\to\infty}\frac{1}{\nu_N}E[Y_N(s_Nt)]=g(t).
\end{align}
Then for all $\eps>0$ and $\delta>0$, we have
$$
\lim_{N\to\infty}P\left(\nu_Ng(t)(1-\eps)\leq Y_N(s_Nt)\leq \nu_Ng(t)(1+\eps) \mbox{ for all }t\in[\delta,\delta^{-1}]\right)=1.
$$
\end{Lemma}
\begin{proof}
Choose $\theta > 0$ sufficiently small that $(1 + 2 \theta)(1 + \theta) \leq 1 + \eps$ and $(1 - 2 \theta)/(1 + \theta) \geq 1 - \eps$.
Because $g$ is continuous, and thus uniformly continuous over compact intervals, we can choose a positive integer $M$ depending on $\delta$ and $\theta$ and positive real numbers $\delta = r_1 < r_2 < \dots < r_M = \delta^{-1}$ such that for $k \in \{1, \dots, M\}$, we have
\begin{equation}\label{gucont}
g(r_{k+1}) \leq (1 + \theta) g(r_k).
\end{equation}
Therefore, \eqref{eq:MonotoneProcessCond1} implies that
\begin{align}
\label{eq:conc2}
\lim_{N \rightarrow \infty} P \big( (1 - \theta) E[Y_N(s_N r_k)] \leq Y_N(s_N r_k) \leq (1 + \theta) E[Y_N(s_N r_k)] \mbox{ for all }k \in \{1, \dots, M\} \big) = 1.
\end{align}
Define the event
$$
A_N(\theta)=\left\{(1-2\theta)\nu_Ng(r_k)\leq Y_N(s_Nr_k)\leq (1+2\theta)\nu_Ng(r_k)\mbox{ for all }k\in\{1,\ldots,M\}\right\},
$$
and note that \eqref{eq:MonotoneProcessCond2} and \eqref{eq:conc2} imply that
\begin{align}
\label{ANlim}
\lim_{N\to\infty} P(A_N(\theta))=1.
\end{align}
Suppose $r_k \leq r \leq r_{k+1}$ for some $k \in \{1, \dots, M-1\}$.  Because $t \mapsto Y_N(t)$ is nondecreasing, on the event $A_N(\theta)$ we have $$Y_N(s_N r) \leq Y(s_N r_{k+1}) \leq (1 + 2 \theta) \nu_N g(r_{k+1}) \leq (1 + 2 \theta)(1 + \theta) \nu_N g(r_k) \leq (1 + \eps)\nu_Ng(r)$$ and
$$Y_N(s_N r) \geq Y(s_N r_k) \geq (1 - 2 \theta)\nu_Ng(r_k) \geq \frac{(1 - 2 \theta)\nu_Ng(r_{k+1})}{1 + \theta} \geq (1 - \eps)\nu_Ng(r).$$
The result of the lemma thus follows from (\ref{ANlim}).
\end{proof}

We finally establish a limit theorem for $\sigma_k$ in Cases 6 and 7 when $k=2$, and Case 2 when $k\geq3$.
Note that when $k\geq 3$, we are assuming that $\mu_i = \mu$ for all $i$.  Recall that $\beta_k$ was defined in (\ref{eq:beta_def}).

\begin{Theo}\label{thm:6and7andCaseII}
For all $t>0$, we have the following three statements.
\begin{enumerate}
\item If $k=2$ and the parameters satisfy the conditions of Case 6 then 
$$
\lim_{N\to\infty}P(\sigma_2>\beta_2 t)=\exp\left(-\frac{\gamma_dt^{d+2}}{(d+1)(d+2)}\right).
$$
\item If $k=2$ and the parameters satisfy the conditions of Case 7 then
$$
\lim_{N\to\infty}P(\sigma_2>\beta_2 t)=\exp\left(-c^{(d+1)/(d+2)}\int_0^t\left(1-\exp\left(-\frac{\gamma_d u^{d+1}}{(d+1)c^{(d+1)/(d+2)}}\right)\right)du\right).
$$
\item If $k\geq 3$ and the parameters satisfy the conditions of Case 2 then 
$$
\lim_{N\to\infty}P(\sigma_k>\beta_{k} t)=\exp\left(-\frac{\gamma_d^{k-1}(d!)^{k-1}t^{d(k-1)+k}}{(d(k-1)+k)!}\right).
$$
\end{enumerate}
\end{Theo}

\begin{proof}
Let $\mathcal{G}_{k-1}$ be the $\sigma$-field generated by the Poisson processes $\Pi_1, \dots, \Pi_{k-1}$, and note that the process $(Y_{k-1}(t), t\geq 0)$ is measurable with respect to $\mathcal{G}_{k-1}$. Therefore,
$$
P(\sigma_k>t|\mathcal{G}_{k-1})=\exp\left(-\mu_k\int_0^tY_{k-1}(s)ds\right).
$$
Next, for a continuous non-negative function $g$, a sequence $(\nu_N)_{N=1}^{\infty}$ of positive numbers, and positive constants $\delta$ and $\eps$, define the $\mathcal{G}_{k-1}$-measurable set
$$
B_{N}^{k-1}(\delta,\eps,g,\nu_N)=\left\{g(u)(1-\eps)\nu_N\leq Y_{k-1}(\beta_ku)\leq g(u)(1+\eps)\nu_N,\mbox{ for all }u\in [\delta,\delta^{-1}]\right\}.
$$
Suppose $r \in [\delta, \delta^{-1}]$.  Observe that on $B_{N}^{k-1}(\delta,\eps,g,\nu_N)$ we have
\begin{align*}
P(\sigma_k>\beta_k r|\mathcal{G}_{k-1})&\leq\exp\left(-\mu_k\int_{\beta_k\delta}^{\beta_k r}Y_{k-1}(s)ds\right)\\
&=
\exp\left(-\mu_k\beta_k\int_{\delta}^{r}Y_{k-1}(\beta_k u)du\right)\\
&\leq
\exp\left(-\mu_k\beta_k\nu_N(1-\eps)\int_{\delta}^rg(u)du\right).
\end{align*}
We thus conclude that
\begin{align}
\label{eq:UpperBound67_II}
P(\sigma_k>\beta_k r)\leq\exp\left(-\mu_k\beta_k\nu_N(1-\eps)\int_{\delta}^rg(u)du\right)+P\left(B_{N}^{k-1}(\delta,\eps,g,\nu_N)^c\right).
\end{align}
To obtain a lower bound, we use the inequality $e^{-x}\geq 1-x$ to get
\begin{align*}
P(\sigma_k>\beta_kr)&= E\left[\exp\left(-\mu_k\int_0^{\beta_k r}Y_{k-1}(s)ds\right)\right]\\
&\geq 
E\left[\left(1- \mu_k \int_0^{\beta_k\delta} Y_{k-1}(s)ds\right)\exp\left(-\mu_k\int_{\beta_k\delta}^{\beta_k r}Y_{k-1}(s)ds\right)\right]\\
&\geq
E\left[\exp\left(-\mu_k\int_{\beta_k\delta}^{\beta_k r}Y_{k-1}(s)ds\right)\right]-E\left[\mu_k\int_0^{\beta_k\delta}Y_{k-1}(s)ds\right].
\end{align*}
Reasoning as in \eqref{eq:UpperBound67_II} we get
$$
E\left[\exp\left(-\mu_k\int_{\beta_k\delta}^{\beta_k r}Y_{k-1}(s)ds\right)\right]\geq P\left(B_N^{k-1}\left(\delta,\eps,g,\nu_N\right)\right)\exp\left(-\nu_N(1+\eps)\beta_k\mu_k\int_{\delta}^rg(u)du\right).
$$
By Lemma \ref{lemma:Occupancy},
$$
E[Y_k(t)] \leq \frac{N \gamma_d^k (d!)^k \alpha^{kd}}{(k(d+1))!} \bigg( \prod_{i=1}^k \mu_i \bigg) t^{k(d+1)}.
$$
Therefore we can use the definition of $\beta_k$ to see that
\begin{align*}
\mu_k\int_0^{\beta_k\delta}E[Y_{k-1}(s)]ds&\leq \bigg( \prod_{i=1}^k \mu_i \bigg)\frac{N \gamma_d^{k-1}  (d!)^{k-1}\alpha^{(k-1)d}}{((k-1)(d+1))! (d(k-1)+k)}\beta_k^{d(k-1)+k}\delta^{d(k-1)+k}\\
&= \frac{\gamma_d^{k-1} (d!)^{k-1}}{(d(k-1)+k)!} \delta^{d(k-1)+k}.
\end{align*}
Therefore, we have the lower bound
\begin{align}
\label{eq:LowerBound67_II}
&P(\sigma_k>\beta_k r) \nonumber \\
&\hspace{.1in}\geq P\left(B_N^{k-1}\left(\delta,\eps,g,\nu_N\right)\right)\exp\left(-\nu_N(1+\eps)\beta_k\mu_k\int_{\delta}^rg(u)du\right)-\frac{\gamma_d^{k-1} (d!)^{k-1}}{(d(k-1)+k)!} \delta^{d(k-1)+k}.
\end{align}
We will now prove that for each of our three scenarios we can choose $\nu_N$ and $g$ such that $\lim_{N\to\infty}\nu_N\beta_k\mu_k$ exists and $P(B_{N}^{k-1}(\delta,\eps,g,\nu_N))$ goes to 1 as $N\to\infty$. We will do this by using Lemma \ref{conc:MonotoneProcess}.
Since $\delta>0$ and $\eps > 0$ are arbitrary, the result will then follow from (\ref{eq:UpperBound67_II}) and (\ref{eq:LowerBound67_II}).

We now prove the three statements of the theorem.  First, suppose $k=2$ and the parameters satisfy the conditions of Case 6.  We set $\nu_N=1/(\beta_2\mu_2)$ and define the function 
$$
g_2(u)=\gamma_d u^{d+1}/(d+1).
$$
To use Lemma \ref{conc:MonotoneProcess}, we first show the hypotheses of Lemma \ref{conc1} are satisfied, that is, $\alpha\beta_2\ll N^{1/d}$ and  $N\mu_1\beta_2\gg 1$.  To show that $\alpha\beta_2\ll N^{1/d}$ note that using the second assumption of Case 6,
$$
\left(\frac{\alpha^d\beta_2^d}{N}\right)^{d+2}=\frac{\alpha^{2d}}{N^{2(d+1)}\mu_1^d\mu_2^d}=\left(\frac{\alpha^d\mu_1}{(N\mu_2)^{d+1}}\right)\left(\frac{\alpha^d\mu_2}{(N\mu_1)^{d+1}}\right)\ll 1.
$$
To show that $N\mu_1\beta_2\gg 1$, use the definition of $\beta_2$ to conclude that
$$
\left(N\mu_1\beta_2\right)^{d+2}=\frac{(N\mu_1)^{d+1}}{\mu_2\alpha^d},
$$
which goes to infinity under the assumptions of Case 6. Therefore, Lemma \ref{conc1} applies if $k=2$ and the parameters satisfy the Case 6 assumptions.  Also, the assumptions of Case 6 imply that
$$(\mu_1 \alpha^d \beta_2^{d+1})^{d+2} = \frac{\mu_1 \alpha^d}{(N \mu_2)^{d+1}} \rightarrow 0,$$ and therefore it follows from Lemma \ref{meanY} that
$$
\lim_{N\to\infty}\beta_2\mu_2E[Y_1(\beta_2 t)]=g_2(t).
$$
Thus, Lemma \ref{conc:MonotoneProcess} applies and we can conclude that for $\delta$ and $\eps$ positive,
$$
\lim_{N\to\infty}P\left(B_N^{k-1}\left(\delta,\eps,g_2,\frac{1}{\beta_2\mu_2}\right)\right)=1,
$$
and the result is proven if $k=2$ and the parameters satisfy the Case 6 assumptions.

Suppose $k=2$ and the parameters satisfy the conditions of Case 7.  We set $\nu_N=N$ and define the function
$$
g^*_2(u)=1-\exp\left(-\frac{\gamma_du^{d+1}}{(d+1) c^{(d+1)/(d+2)}}\right).
$$
To use Lemma \ref{conc:MonotoneProcess}, we first show the hypotheses of Lemma \ref{conc1} are satisfied, that is,
$\alpha\beta_2\ll N^{1/d}$ and  $N\mu_1\beta_2\gg 1$. The assumptions of Case 7 imply that $N \mu_2 \asymp (\mu_1 \alpha^d)^{1/(d+1)}$ and therefore $\beta_2 \asymp (\mu_1 \alpha^d)^{-1/(d+1)}$.  It follows that under the assumptions of Case 7,
$$N \mu_1 \beta_2 \asymp N \mu_1^{d/(d+1)} \alpha^{-d/(d+1)} \rightarrow \infty$$
and $$\alpha \beta_2 \asymp \mu_1^{-1/(d+1)} \alpha^{1/(d+1)} \ll N^{1/d}.$$
We thus conclude that Lemma \ref{conc1} applies.  By Lemma \ref{meanY} and the assumptions of Case 7,
$$
\lim_{N \rightarrow \infty} N^{-1} E[Y_1(\beta_2 r)] = \lim_{N \rightarrow \infty} \bigg( 1 - \exp \bigg(-\frac{\mu_1 \gamma_d \alpha^d (N \mu_1 \mu_2 \alpha^d)^{-(d+1)/(d+2)} r^{d+1}}{d+1} \bigg) \bigg) = g_2^*(r).
$$
We can thus apply Lemma \ref{conc:MonotoneProcess} to conclude that for $\delta > 0$ and $\eps > 0$,
$$
\lim_{N\to\infty}P\left(B_N^{k-1}\left(\delta,\eps,g_2^*,N\right)\right)=1,
$$
when the parameters satisfy the conditions of Case 7. To conclude the proof in this setting note that the assumptions of Case 7 imply
$$
\lim_{N\to\infty}N\beta_2\mu_2=c^{(d+1)/(d+2)}.
$$

We next consider $k \geq 3$ and parameters that satisfy Case 2, so that $\mu_i = \mu$ for all $i$. In this case we set $\nu_N=1/(\beta_k\mu)$ and set $g=g_k$, where
$$
g_k(t)=\frac{\gamma_d^{k-1}(d!)^{k-1}t^{(k-1)(d+1)}}{((k-1)(d+1))!}.
$$
From the definitions of $\beta_k$ and $v_{k-1}(t)$ it is immediate that
$$
\beta_k\mu v_{k-1}(\beta_k t)=g_k(t).
$$
A short calculation yields
\begin{equation}\label{tozero}
\mu \alpha^d \beta_k^{d+1} = \left(\mu^{-d} N^{-(d+1)} \alpha^d \right)^{1/((k-1)d + k)} \rightarrow 0.
\end{equation}
Thus, we can apply Lemma \ref{thm:no_overlaps} to see that as $N\to\infty$, we have $v_{k-1}(\beta_k t)\sim E[Y_{k-1}(\beta_k t)]$ and therefore
$$
\lim_{N\to\infty}\beta_k\mu E[Y_{k-1}(\beta_k t)]=g_k(t).
$$
Application of Lemma \ref{conc:MonotoneProcess} thus requires that we prove \eqref{eq:MonotoneProcessCond1} for the process $(Y_{k-1}(t), t \geq 0)$ with $s_N=\beta_k$ and $\nu_N=1/(\beta_k\mu)$.  By Lemma \ref{conck}, we need to check (\ref{largemucond}), which in this case means showing that $N \mu^{k-1} \alpha^{(k-2)d} \beta_k^{(k-2)d + k-1} \rightarrow \infty$.  Using (\ref{tozero}),
$$N \mu^{k-1} \alpha^{(k-2)d} \beta_k^{(k-2)d + k - 1} = \frac{N \mu^{k-1} \alpha^{(k-1)d} \beta_k^{(k-1)d + k}}{\alpha^d \beta_k^{d+1}} = \frac{1}{\mu \alpha^d \beta_k^{d+1}} \rightarrow \infty.$$
Thus, Lemma \ref{conc:MonotoneProcess} implies that $P(B_N^{k-1}(\delta,\eps,1/(\beta_k\mu),g_k))\to 1$ as $N\to\infty$, which completes the proof.
\end{proof}

We finish this section with a limit theorem for $\sigma_2$ in Case 8.

\begin{Theo}\label{thm:sigma2Case8}
Assume the conditions of Case 8 hold. Then for $t>0$,
$$\lim_{N\to\infty}P(N\mu_2\sigma_2>t)=e^{-t}.$$
\end{Theo}

\begin{proof}
Fix $\eps>0$ and define
$$
t_1(\eps)=\left(\frac{(d+1)\log(1/\eps)}{\mu_1\alpha^d\gamma_d}\right)^{1/(d+1)}.
$$
Note that under the assumptions of Case 8, $t_1(\eps)\ll N^{1/d}/\alpha$, and we can therefore apply Lemma \ref{meanY} to see that $E[Y_1(t_1(\eps))]=N(1-\eps)$.

We next will use Lemma \ref{conc1} to show that with high probability $Y_1(t_1(\eps))\geq N(1-\eps)^2$. We verify the first condition of Lemma \ref{conc1} by noting that, by the first condition of Case 8,
$$
N\mu_1t_1(\eps) \asymp \frac{N\mu_1^{d/(d+1)}}{\alpha^{d/(d+1)}} = \left(\frac{N^{(d+1)/d}\mu_1}{\alpha}\right)^{d/(d+1)} \rightarrow \infty.
$$
The second condition of Lemma \ref{conc1} is satisfied because $t_1(\eps) \ll N^{1/d}/\alpha$ as noted above.
Thus Lemma \ref{conc1} implies that
$$
\lim_{N\to\infty}P\left(Y_1(t_1(\eps))\geq N(1-\eps)^2\right)=1.
$$
In addition, if we define the event $D_\eps(N)=\{Y_1(t)\geq (1-\eps)^2 N \mbox{ for all } t\geq t_1(\eps)\}$, then the monotonicity of $Y_1$ gives us that $P(D_\eps(N))\to 1$ as $N\to\infty$.

We now write
$$
N\mu_2\sigma_2=N\mu_2t_1(\eps)+N\mu_2(\sigma_2-t_1(\eps)).
$$
We will show that the first term on the right hand side converges to zero and that the second term converges to an Exponential(1) random variable. Note that from the second condition of Case 8 that
$$
\frac{N\mu_2}{\left(\mu_1\alpha^d\right)^{1/(d+1)}}\ll 1,
$$
and thus $N\mu_2t_1(\eps)\ll 1.$ Finally consider
\begin{align*}
P\left(N\mu_2\left(\sigma_2-t_1(\eps)\right)>t\right)&=E\left[\exp\left(-\mu_2\int_0^{t_1(\eps)+t/(N\mu_2)}Y_1(s) \: ds\right)\right]\\
&\leq
E\left[\exp\left(-\mu_2\int_{t_1(\eps)}^{t_1(\eps)+t/(N\mu_2)}Y_1(s) \: ds\right)\right]\\
&=
E\left[\exp\left(-\mu_2\int_{t_1(\eps)}^{t_1(\eps)+t/(N\mu_2)}Y_1(s) \: ds\right)1_{D_\eps(N)}\right]\\
&\quad+
E\left[\exp\left(-\mu_2\int_{t_1(\eps)}^{t_1(\eps)+t/(N\mu_2)}Y_1(s) \: ds\right)1_{D_\eps(N)^c}\right]\\
&\leq
\exp\left(-(1-\eps)^2t\right)+P(D_\eps(N)^c).
\end{align*}
Next, we get the lower bound
$$
P\left(N\mu_2\left(\sigma_2-t_1(\eps)\right)>t\right) = E \left[ \exp\left(-\mu_2\int_0^{t_1(\eps)+t/(N\mu_2)}Y_1(s)ds\right) \right] \geq
e^{-t}\exp\left(-\mu_2Nt_1(\eps)\right).
$$
We thus conclude that
$$
e^{-t}\leq\liminf_{N\to\infty}P\left(N\mu_2\left(\sigma_2-t_1(\eps)\right)>t\right)\leq\limsup_{N\to\infty}P\left(N\mu_2\left(\sigma_2-t_1(\eps)\right)>t\right)\leq e^{-t(1-\eps)^2},
$$
and since $\eps>0$ is arbitrary the result follows.
\end{proof}

\subsection{Proof of Case 3 when $k \geq 3$ and Case 11 when $k = 2$}

Given positive integers $d$ and $k$ and positive real numbers $c_1, \dots, c_k$, we define a random variable $Z_{d,k}(c_1, \dots, c_k)$ that has the same distribution as $\sigma_k$ when $L = 1$, $\alpha = 1$, and $\mu_i = c_i$ for all $i \in \{1, \dots, k\}$.  That is, we assign a type to each site on the torus $[0,1]^d$.  At time zero, all sites have type 0.  At the times and locations of a homogeneous Poisson process of rate $c_1$ per unit area, a mutation to type 1 occurs, causing a region of type 1 individuals to grow outward from this point at rate one.  That is, $t$ time units after the mutation, the radius of the type 1 region resulting from the mutation will be $t$.  Type 1 sites acquire a second mutation at rate $c_2$ per unit area, causing a region of type 2 individuals to grow outward from this point at rate one.  This process continues until some site has type $k$.  Then $Z_{d,k}(c_1, \dots, c_k)$ denotes the first time that some site has type $k$. 

Theorem \ref{k3case3} below describes the asymptotic behavior of $\sigma_k$ in Case 3 when $k \geq 3$.  Note that under the assumptions of Case 3 when $k \geq 3$, we have $\alpha/N^{1/d} \asymp N \mu$, so the scaling in Theorem~\ref{k3case3} is comparable to the scaling in (\ref{3Xeq}).  As can be seen from Corollary \ref{case11} below, Theorem \ref{k3case3} also implies the result for Case 11 when $k = 2$.

\begin{Theo}\label{k3case3}
Fix a positive integer $k \geq 2$.  Suppose there are positive real numbers $c_1, \dots, c_k$ such that as $N \rightarrow \infty$, we have
$$\frac{\mu_i N^{(d+1)/d}}{\alpha} \rightarrow c_i$$ for all $i \in \{1, \dots, k\}$.  Then, as $N \rightarrow \infty$, we have
$$\frac{\alpha \sigma_k}{N^{1/d}} \Rightarrow Z_{d,k}(c_1, \dots, c_k).$$
\end{Theo}

\begin{proof}
Consider a rescaling of the original process defined on the torus $[0, 1]^d$ of side length $1$ such that the type of the site $x$ in the rescaled process at time $t$ is the same as the type of the original process at the site $N^{1/d} x$ at time $N^{1/d} t/\alpha$.  After this time and space rescaling, the radius of a region of type $i$ individuals is expanding at rate $$\alpha \cdot \frac{1}{N^{1/d}} \cdot \frac{N^{1/d}}{\alpha} = 1,$$ where the second factor on the left-hand side accounts for the rescaling of space and the third factor accounts for the rescaling of time.  For this rescaled process, the rate of type $i$ mutations per unit volume is given by $$\mu_i \cdot N \cdot \frac{N^{1/d}}{\alpha} = \frac{\mu_i N^{(d+1)/d}}{\alpha}.$$
Therefore, the distribution of the time before a type $k$ site appears in the rescaled process is exactly the same as the distribution of $$Z_{d,k}\bigg( \frac{\mu_1 N^{(d+1)/d}}{\alpha}, \dots, \frac{\mu_k N^{(d+1)/d}}{\alpha} \bigg).$$  By assumption, we have $\mu_i N^{(d+1)/d}/\alpha \rightarrow c_i$ as $N \rightarrow \infty$ for all $i \in \{1, \dots, k\}$.  Because it is easy to see that for any positive integers $d$ and $k$, the distribution of $Z_{d,k}(c_1, \dots, c_k)$ is a continuous function of $c_1, \dots, c_k$, it follows that the distribution of the time until some individual acquires $k$ mutations in the rescaled process converges as $N \rightarrow \infty$ to the distribution of $Z_{d,k}(c_1, \dots, c_k)$.  This observation implies the result, after taking into account the rescaling of time.
\end{proof}

\begin{Cor}\label{case11}
Suppose there are positive real numbers $c_1$ and $c_2$ such that as $N \rightarrow \infty$, we have
\begin{equation}\label{case11cond}
\frac{\mu_1 N^{(d+1)/d}}{\alpha} \rightarrow c_1, \hspace{.5in} \frac{\mu_2 \alpha^d}{(N \mu_1)^{d+1}} \rightarrow c_2.
\end{equation}
Then, as $N \rightarrow \infty$, we have $$\frac{\alpha}{N^{1/d}} \sigma_2 \Rightarrow Z_{d,2}(c_1, c_2 c_1^{d+1}).$$
\end{Cor}

\begin{proof}
Note that (\ref{case11cond}) implies that
$$\frac{\mu_2 N^{(d+1)/d}}{\alpha} \rightarrow c_2 c_1^{d+1}.$$  The result therefore follows from Theorem \ref{k3case3} when $k = 2$.
\end{proof}

The last two propositions collect facts about the random variables $Z_{d,k}(c_1, \dots, c_k)$.

\begin{Prop}\label{distprop1}
Suppose $W_1, \dots, W_k$ are independent exponentially distributed random variables with rate parameters $c_1, \dots, c_k$ respectively.  We write $X \preceq Y$ if $Y$ stochastically dominates $X$.  Then
\begin{equation}\label{sdom}
\sum_{i=1}^k W_i \preceq Z_{d,k}(c_1, \dots, c_k) \preceq \frac{(k-1) \sqrt{d}}{2} + \sum_{i=1}^k W_i.
\end{equation}
\end{Prop}

\begin{proof}
We first obtain a stochastic lower bound for $Z_{d,k}(c_1, \dots, c_k)$.
Consider a modification of the process such that as soon as the first type $i$ mutation occurs, every site instantly becomes type $i$.  This change can only reduce the time until some site acquires a $k$th mutation.  In this modified model, the distribution of the time until some individual acquires $k$ mutations is exactly the distribution of $\sum_{i=1}^k W_i$.  The lower bound in (\ref{sdom}) follows.

For the stochastic upper bound, consider a different modification of the process in which, once the first type $i$ mutation occurs, all further mutations are suppressed until every site has type $i$.  This change can only increase the time until some site acquires a $k$th mutation.  For all $x, y \in [0,1]^d$, we have $|x - y| \leq \frac{1}{2}\sqrt{d}$.  Therefore, because the radius of the region of type $i$ individuals increases at unit speed after the mutation, it takes a time of exactly $\frac{1}{2}\sqrt{d}$ for this mutation to spread to the entire torus.  Therefore, in this modified model, the distribution of the time until some individual acquires two mutations is exactly the distribution of $\frac{k-1}{2}\sqrt{d} + \sum_{i=1}^k W_i$, as we must take into account not only the waiting times for the $k$ mutations but also the times for the first $k-1$ of these mutations to spread to the rest of the torus.  This observation implies the upper bound in (\ref{sdom}).
\end{proof}

\begin{Prop}\label{distprop2}
For all positive integers $d$ and $k$ and all positive real numbers $c_1, \dots, c_k$, we have
$$\lim_{t \rightarrow 0} t^{-((k-1)d +k)} P(Z_{d,k}(c_1, \dots, c_k) \leq t) = \frac{(d!)^{k-1} \gamma_d^{k-1} c_1 \dots c_k}{((k-1)d + k)!}.$$
\end{Prop}

\begin{proof}
We will consider the process in which $L = 1$, $\alpha = 1$, and $\mu_i = c_i$ for all $i \in \{1, \dots, k\}$.  We first obtain an upper bound on the probability that $k$ mutations occur by time $t$.  Let $m_k(t)$ be the mean volume of the region with type $k$ individuals at time $t$.  
The expected rate of mutations to type $k$ at time $u$ is bounded above by $c_k m_{k-1}(u)$, and if such a mutation occurs at time $u$, then the volume of the type $k$ region created by this mutation at time $t$ will be at most $\gamma_d (t - u)^d$.  It follows that
$$m_k(t) \leq \int_0^t c_k m_{k-1}(u) \gamma_d (t - u)^d \: du.$$  Also, because the entire torus has volume $1$, we know that $m_0(t) \leq 1$.  Therefore, by the same inductive argument used to establish (\ref{vkdef}) with $N = 1$, $\alpha = 1$, and $\mu_i = c_i$ for all $i$, we get that for all positive integers $k$,
\begin{equation}\label{mjt}
m_k(t) \leq \frac{(d!)^k \gamma_d^k c_1 \dots c_k t^{k(d+1)}}{(k(d+1))!}.
\end{equation}
Therefore, the expected number of mutations to type $k$ that occur by time $t$ is $$\int_0^t c_k m_{k-1}(u) \: du \leq \frac{(d!)^{k-1} \gamma_d^{k-1} c_1 \dots c_k t^{(k-1)d + k}}{((k-1)d + k)!}.$$  It now follows from Markov's Inequality that
\begin{equation}\label{Zkupper}
P(Z_{d,k}(c_1, \dots, c_k) \leq t) \leq \frac{(d!)^{k-1} \gamma_d^{k-1} c_1 \dots c_k t^{(k-1)d + k}}{((k-1)d + k)!}.
\end{equation}

It remains to obtain a lower bound.  For the lower bound, we will consider a modified process in which, for all $j \geq 1$, only the first mutation to type $j$ is permitted.  This modification can only reduce the probability of observing a mutation to type $k$ by time $t$.  We will assume that $t < 1/2$, so that if a mutation appears at time $s < t$, the volume of the region to which the mutation has spread by time $u$, where $s < u < t$, is exactly $\gamma_d(u - s)^d$.  Let $\sigma_j^*$ denote the time at which the $j$th mutation appears in this modified model.  Let $\tau_1 = \sigma_1^*$, and for $j \geq 2$, let $\tau_j = \sigma^*_j - \sigma^*_{j-1}$.  Note that $$P(Z_{d,k}(c_1, \dots, c_k) \leq t) \geq P(\sigma_k^* \leq t) = P(\tau_1 + \dots + \tau_k \leq t).$$  Now $\tau_1$ has an exponential distribution with rate parameter $c_1$, so its probability density function is $$f_1(u) = c_1 e^{-c_1 u}, \hspace{.2in}u > 1/2.$$  For $j \geq 2$ and $u < 1/2$, we have
$$P(\tau_j > u) = \exp \bigg( - c_j \int_0^u \gamma_d (u - s)^d \: ds \bigg) = \exp \bigg(- \frac{c_j \gamma_d u^{d+1}}{d+1} \bigg).$$
Therefore, denoting by $f_j$ the probability density function of $\tau_j$, we have for $0 < u < 1/2$,
$$f_j(u) = c_j \gamma_d u^d \exp \bigg( - \frac{c_j \gamma_d u^{d+1}}{d+1} \bigg).$$  Furthermore, the random variables $\tau_1, \dots, \tau_k$ are independent.  Therefore,
$$P(\tau_1 + \dots + \tau_k \leq t) \geq \int_0^t \int_0^{t - s_1} \ldots \int_0^{t - s_1 - \ldots - s_{k-1}} f_1(s_1) \dots f_k(s_k) \: ds_k \dots ds_1.$$
Writing $$J_k(t) = e^{-c_1 t} \prod_{j=2}^k \exp \bigg( - \frac{c_j \gamma_d t^{d+1}}{d+1} \bigg),$$ it follows that
$$P(\tau_1 + \dots + \tau_k \leq t) \geq J_k(t) \gamma_d^{k-1} c_1 \dots c_k \int_0^t \int_0^{t - s_1} \ldots \int_0^{t - s_1 - \ldots - s_{k-1}} s_2^d \dots s_k^d \: ds_k \dots ds_1.$$
A tedious calculation yields that the $k$-fold integral above equals $$\frac{(d!)^{k-1} t^{(k-1)d + k}}{((k-1)d + k)!}.$$  Therefore,
\begin{equation}\label{Zklower}
P(Z_{d,k}(c_1, \dots, c_k) \leq t) \geq \frac{J_k(t) (d!)^{k-1} \gamma_d^{k-1} c_1 \dots c_k t^{(k-1)d + k}}{((k-1)d + k)!}.
\end{equation}
Because $$\lim_{t \rightarrow 0} J_k(t) = 1,$$ the result follows from (\ref{Zkupper}) and (\ref{Zklower}).
\end{proof}


\begin{thebibliography}{99}
\bibitem{armdoll} P. Armitage and R. Doll (1954).  The age distribution of cancer and a multi-stage theory of carcinogenesis. {\it Brit. J. Cancer} {\bf 8}, 1-12.

\bibitem{bg1} M. Bramson and D. Griffeath (1980).  On the Williams-Bjerknes tumour growth model: II.  {\it Math. Proc. Camb. Phil. Soc.} {\bf 88}, 339-357.

\bibitem{bg2} M. Bramson and D. Griffeath (1981).  On the Williams-Bjerknes tumour growth model: I.  {\it Ann. Probab.} {\bf 9}, 173-185.

\bibitem{dfl16} R. Durrett, J. Foo, and K. Leder (2016).  Spatial Moran models II: cancer initiation in spatially structured tissue.  {\it J. Math. Biol.} {\bf 72}, 1369-1400.

\bibitem{dflmm}R. Durrett, J. Foo, K. Leder, J. Mayberry, and F. Michor (2010).  Evolutionary dynamics of tumor progression with random fitness values.  {\it Theor. Pop. Biol.} {\bf 78}, 54-66.

\bibitem{dm11} R. Durrett and J. Mayberry (2011).  Traveling waves of selective sweeps.  {\it Ann. Appl. Probab.} {\bf 21}, 699-744.

\bibitem{dm10} R. Durrett and S. Moseley (2010).  Evolution of resistance and progression to disease during clonal expansion of cancer.  {\it Theor. Pop. Biol.} {\bf 77}, 42-48.

\bibitem{dm15} R. Durrett and S. Moseley (2015).  Spatial Moran models I: stochastic tunneling in the neutral case.  {\it Ann. Appl. Probab.} {\bf 25}, 104-115.

\bibitem{dss} R. Durrett, D. Schmidt, and J. Schweinsberg (2009).  A waiting time problem arising from the study of multi-stage carcinogenesis.  {\it Ann. Appl. Probab.} {\bf 19} (2009), 676-718.

\bibitem{flr14} J. Foo, K. Leder, and M. D. Ryser (2014).  Multifocality and recurrence risk: a quantitative model of field cancerization.  {\it J. Theor. Biol.} {\bf 355}, 170-184.

\bibitem{imkn} Y. Iwasa, F. Michor, N. L. Komarova, and M. A. Nowak (2005).  Population genetics of tumor suppressor genes.  {\it J. Theor. Biol.} {\bf 233}, 15-23.

\bibitem{imn} Y. Iwasa, F. Michor, and M. A. Nowak (2004).  Stochastic tunnels in evolutionary dynamics.  {\it Genetics} {\bf 166}, 1571-1579.

\bibitem{k06} N. L. Komarova (2006).  Spatial stochastic models for cancer initiation and progression.  {\it Bull. Math. Biol.} {\bf 68}, 1573-1599.

\bibitem{ksn} N. L. Komarova, A. Sengupta, and M. A. Nowak (2003).  Mutation-selection networks of cancer initiation: tumor suppressor genes and chromosomal instability.  {\it J. Theor. Biol.} {\bf 223}, 433-450.

\bibitem{mh11} E. A. Martens and O. Hallatschek (2011).  Interfering waves of adaptation promote spatial mixing.  {\it Genetics} {\bf 189}, 1045-1060.

\bibitem{mkmh11} E. A. Martens, R. Kostadinov, C. C. Maley, and O. Hallatschek (2011).  Spatial structure increases the waiting time for cancer.  {\it New J. Physics} {\bf 13}, 115014.

\bibitem{moll72} D. Mollison (1972). Conjecture on the spread of infection in two dimensions disproved. {\it Nature} {\bf 240}, 467-468.

\bibitem{rc10} P. Ralph and G. Coop (2010).  Parallel adaptation: one or many waves of advance of an advantageous allele.  {\it Genetics} {\bf 186}, 647-668.

\bibitem{rlrlf16} M. Ryser, W. Lee, N. Ready, K. Leder, J. Foo (2016). Quantifying the Dynamics of Field Cancerization in Tobacco-Related Head and Neck Cancer: A Multiscale Modeling Approach. {\it Cancer Research} {\bf 76}, 7078-7088.

\bibitem{sch08} J. Schweinsberg (2008).  Waiting for $m$ mutations.  {\it Electron. J. Probab.} {\bf 13}, 1442-1478.

\bibitem{wb72} T. Williams and R. Bjerknes (1972). Stochastic model for abnormal clone spread through epithelial basal layer. {\it Nature} {\bf 236}, 19-21.
\end{thebibliography}
\end{document}